\documentclass[a4paper,11 pt,twoside]{amsart}

\usepackage[utf8]{inputenc}
\usepackage[english]{babel}
\usepackage{times}
\usepackage{amsmath}
\usepackage{amsfonts}
\usepackage{amssymb}
\usepackage{amsmath}
\usepackage{amsthm}
\usepackage{enumerate}
\usepackage{hyperref}
\usepackage{fancyhdr}
\usepackage{mathrsfs} 
\usepackage{tikz}

\usepackage{stmaryrd}
\renewcommand{\dots}{\ifmmode\mathinner{\ldotp\kern-0.2em\ldotp\kern-0.2em\ldotp}\else.\kern-0.13em.\kern-0.13em.\fi}

\newtheorem{theorem}{Theorem}
\newtheorem{lemma}{Lemma}
\newtheorem{proposition}[lemma]{Proposition}

\newtheorem{remark}[lemma]{Remark}

\newtheorem{conjecture}[lemma]{Conjecture}

\definecolor{darkgreen}{rgb}{0.1,0.7,0.1}

\definecolor{darkred}{rgb}{0.7,0.1,0.1}
\newcommand{\E}{\mathbb{E}}
\renewcommand{\P}{\mathbb{P}}

\newcommand{\bH}{\mathbf{H}}

\newcommand{\bbE}{\mathbb{E}}

\newcommand{\bbN}{\mathbb{N}}

\newcommand{\bbP}{\mathbb{P}}

\newcommand{\bbR}{\mathbb{R}}

\newcommand{\bbZ}{\mathbb{Z}}

\newcommand{\cA}{\mathcal{A}}

\newcommand{\cC}{\mathcal{C}}
\newcommand{\cD}{\mathcal{D}}

\newcommand{\cF}{\mathcal{F}}

\newcommand{\cH}{\mathcal{H}}

\newcommand{\cL}{\mathcal{L}}

\newcommand{\cP}{\mathcal{P}}
\newcommand{\cQ}{\mathcal{Q}}

\newcommand{\cT}{\mathcal{T}}

\newcommand{\cW}{\mathcal{W}}

\newcommand{\gb}{\beta}
\newcommand{\gd}{\delta}
\newcommand{\gep}{\varepsilon}       

\newcommand{\gD}{\Delta}

\newcommand{\gO}{\Omega}
\newcommand{\gl}{\lambda}

\newcommand{\ind}{\mathbf{1}}
\DeclareMathOperator{\gap}{\mathrm{gap}}
\newcommand{\lint}{\llbracket}
\newcommand{\rint}{\rrbracket}

\DeclareMathSymbol{\leqslant}{\mathalpha}{AMSa}{"36} 
\DeclareMathSymbol{\geqslant}{\mathalpha}{AMSa}{"3E} 
\DeclareMathSymbol{\eset}{\mathalpha}{AMSb}{"3F}     
\newcommand{\dd}{\,\text{\rm d}}             

\newcommand{\var}{{\rm Var}}
\newcommand{\Tm}{T_{\rm mix}}
\newcommand{\cc}{\complement}

\renewcommand{\tilde}{\widetilde}
          
\newtheorem{teorema}{Theorem}

\begin{document}

\title{Mixing time and cutoff for the weakly asymmetric\\
simple exclusion process}

\author{Cyril Labb\'e}
\address{Universit\'e Paris-Dauphine, PSL Research University, Ceremade, CNRS, 75775 Paris Cedex 16, France.}
\email{labbe@ceremade.dauphine.fr}
\author{Hubert Lacoin}
\address{IMPA, Estrada Dona Castorina 110, Rio de Janeiro, Brasil.}
\email{lacoin@impa.br}

\pagestyle{fancy}
\fancyhead[LO]{}
\fancyhead[CO]{\sc{C.~Labb\'e and H.~Lacoin}}
\fancyhead[RO]{}
\fancyhead[LE]{}
\fancyhead[CE]{\sc{Cutoff for the WASEP}}
\fancyhead[RE]{}

\date{\small\today}

\begin{abstract}

We consider the simple exclusion process with $k$ particles on a segment of length $N$ performing random walks with transition $p>1/2$ to the right and
$q=1-p$ to the left.
We focus on the case where the asymmetry in the jump rates $b=p-q>0$ vanishes in the limit when $N$ and $k$ tend to infinity, 
and obtain sharp asymptotics for the mixing times of this sequence of Markov 
chains in the two cases where the asymmetry is either much larger or much smaller than $(\log k)/N$. We show that in the former case ($b \gg (\log k)/N$), 
the mixing time corresponds to the time needed to reach macroscopic equilibrium,  
like for the strongly asymmetric (i.e.\ constant $b$) case studied in~\cite{LabLac16}, 
while the latter case ($b \ll (\log k)/N$) macroscopic equilibrium is not sufficient for mixing and one must wait till local fluctuations equilibrate,
similarly to what happens in the symmetric case worked out in \cite{Lac16}. In both cases, convergence to equilibrium is abrupt: we have a cutoff phenomenon for the total-variation distance. 
We present a conjecture for the remaining regime when the asymmetry is of order $(\log k) / N$.

\medskip

\end{abstract}

\maketitle

\setcounter{tocdepth}{1}
\tableofcontents

\newpage 

\section{Introduction}

The simple exclusion process is a model of statistical mechanics that provides a simplified picture for a gas of interacting particles.
Particles move on a lattice, each of them performing a nearest neighbor random walk independently of the others, and interact only 
via the exclusion rule that prevents any two particles from sharing the same site (when a particle tries to jump on a site which is already occupied, this jump is cancelled).

\medskip

In spite of its simplicity, this model displays a very rich behavior and has given rise to a rich literature both in theoretical physics and mathematics, see for instance~\cite{KipLan,Liggett} and references therein.

\medskip

In the present paper, we study relaxation to equilibrium for a particular instance of the simple exclusion process in which the lattice is a segment of length $N$ and particles feel a bias towards the right that vanishes when $N$ tends to infinity.
This setup is often referred to as the Weakly Asymmetric Simple Exclusion Process (WASEP): it interpolates between the symmetric case
(SSEP)
and the one with a positive constant bias (ASEP).

\medskip

While convergence to equilibrium for a particle system can be considered on a macroscopic scale via the evolution of the particle density
or hydrodynamic profile (see e.g. \cite{KipLan} and references therein), 
an alternative and complementary viewpoint (when the system is of finite size) consists in measuring the so-called 
$\gep$-Total Variation Mixing Time \cite{LevPerWil}.
It is defined as the first time 
at which the total variation distance to the stationary state, starting from the ``worst" initial condition, falls below a given threshold $\epsilon$. Compared to the hydrodynamic profile, this provides a much more microscopic information on the particle system.

\medskip

The problem of mixing time of the simple exclusion process on the segment has been extensively studied both in the symmetric \cite{Wil04, Lac16, Lac162} and the asymmetric setup 
\cite{Benjamini, LabLac16} and it has been proved in \cite{Lac16} and \cite{LabLac16} respectively that in both cases,
the worst case total variation distance drops abruptly from its maximal value $1$ to $0$, so that the mixing time does not depend at first order on the choice of the
threshold $\gep$ - a phenomenon known as cutoff and conjectured to hold for a large class of Markov chains as soon as the mixing time is of a larger order 
than the relaxation time (which is defined as the inverse of the spectral gap of the generator).

\medskip

However the patterns of convergence to equilibrium in the symmetric and asymmetric cases are very different. Let us for simplicity focus on the case
with a density of particles $k=\alpha N$, $\alpha\in (0,1)$.
In the symmetric case, the time scale associated with the hydrodynamic profile is $N^2$ and the limit is given by the heat equation \cite{KOV} (which takes an infinite time 
to relax to its equilibrium profile which is flat) and microscopic mixing occurs on a larger time scale $N^2\log N$.

\medskip

In the asymmetric setup the hydrodynamic limit is given by the inviscid Burgers' equation with a shorter time scale $N$ \cite{Reza} 
(see also \cite{LabbeKPZ, LabLac16} for adaptations 
of this result to the segment).
The equilibrium profile for this equation is reached after a finite time and in this case, the mixing time is of order $N$ and corresponds exactly to the 
time at which macroscopic equilibrium is attained.

\medskip

The aim of this paper is to understand better the role of the asymmetry in mixing and how one interpolates between the symmetric and asymmetric regimes. This leads us to consider a model with an asymmetry that vanishes with the scale of observation, usually referred to as 
Weakly Asymmetric Exclusion Process (WASEP). While hydrodynamic limit~\cite{Demasi89, Gartner88, KipLan} and 
fluctuations scaling limits \cite{DG91,BG97,LabbeKPZ} for WASEP are now well understood, much less is known about 
how a weak asymmetry affects the mixing time of the system.

\medskip

A first step in this direction was made in \cite{LevPer16}. Therein the order of magnitude for the mixing time was identified for all possible intensities of vanishing bias, but with different constant for the upper and the lower bounds.
Three regimes where distinguished (in the case where there is a density of particles): 
\begin{itemize}
 \item [(A)] When $b_N\le 1/N$, the mixing time remains of the same order as that of the symetric case $N^2\log N$.
  \item [(B)] When $1/N\le b_N\le (\log N)/N$, the mixing time is of order $(b_N)^{-2}\log N$.
  \item [(C)] When $(\log N)/N\le b_N\le 1$, the mixing time is of order $(b_N)^{-1} N$.
\end{itemize}
The transition occurring around $b_N\approx N^{-1}$ is the one observed for the hydrodynamic limit: 
It corresponds to a crossover regime where the limit is given by a viscous Burger's equation \cite{Demasi89, Gartner88, KipLan} 
which interpolates between the heat and the inviscid Burgers' equations.
The one occurring for $b_N\approx N^{-1}\log N$ is however not observed in the macrospic profile and is specific to mixing times.

\medskip

In the present work, we identify the full asymptotic of the mixing time (with the right constant) when the bias is either negligible compared to, 
or much larger than $\log N / N$ (or $\log k/N$ when the total number of particle is not of order $N$). This implies
cutoff in these two regimes. Our result and its proof provide a better understanding of the effect of asymmetry on microscopic mixing:
When $b_N\gg N^{-1}\log N$, the pattern of relaxation is identical to that of the fully asymmetric case and microscopic equilibrium is
reached exactly when the macroscopic profile hits its equilibrium state.
When $b_N\ll N^{-1}\log N$ the pattern of relaxation resembles that of the symmetric case, the mixing time corresponds to the time needed 
to equilibrate local fluctuations, in particular in the case $(B)$ described above (or more precisely when  $1/N\ll b_N\ll (\log N)/N$)
this time does not correspond to the time needed to reach macroscopic equilibrium.

\medskip

We could not prove such a result in the crossover regime $b_N\approx N^{-1}\log N$:
In this case the time to reach macroscopic equilibrium and that to equilibrate local fluctuations are of the same order and 
the two phenomena are difficult to separate.
In Section \ref{conjectos}  we provide a conjecture for the mixing time in this regime in the case of vanishing density. However the techniques developed here are not sufficient to obtain sharp results in this case.

\section{Model and results}

\subsection{Mixing time for the WASEP}

Given $N\in \mathbb N$, $k\in \lint 1,N-1 \rint$ (we use the notation $\lint a,b\rint=[a,b]\cap \bbZ$) and $p\in(1/2,1]$, the Asymmetric Simple Exclusion Process on  $\lint 1, N  \rint$ with $k$ particles and parameter $p$ is the random process on the state space
$$\gO_{N,k}^0:=\Big\{ \xi\in\{0,1\}^{N} \ : \ \sum_{x=1}^N \xi(x)=k\Big\},$$
associated with the generator
\begin{equation}\label{defgen}
 \cL_{N,k} f(\xi):= \sum_{y=1}^{N-1} \left( q\ind_{\{\xi(y)< \xi(y+1)\}}+p\ind_{\{\xi(y)> \xi(y+1)\}}\right)(f(\xi^y)-f(\xi)),
\end{equation}
where $q=1-p$ and
\begin{equation}\label{flipz}
 \xi^y(x):=\begin{cases} \xi(y+1) \quad &\text{ if } x=y,\\
                         \xi(y)  \quad &\text{ if } x=y+1,\\
                         \xi(x) \quad& \text{ if } x\notin \{y,y+1\}.
           \end{cases}
\end{equation}
In a more intuitive manner we can materialize the positions of $1$ by particles, and say that the particles perform random walks with 
jump rates $p$ to the right and $q=1-p$ to the left: These random walks are independent from one another except that any jump that would put a particle at a location already occupied by another particle is
cancelled.
Having in mind this particle representation, we let for $i\in \lint 1, k\rint$, $\xi_i$ denote the position of the $i$-th leftmost particle
$$\xi_i:= \min\left\{ y\in \lint 1,N \rint \ : \ \sum_{x=1}^y \xi(x)=i \right\}.$$

We let $P^{N,k}_t$ denote
the associated semi-group and $(\eta^\xi(t,\cdot))_{t\ge 0}$ denote the trajectory of the Markov chain starting from initial condition $\xi \in \gO_{N,k}^0$.
This Markov chain is irreducible, and admits a unique invariant (and reversible) probability measure $\pi_{N,k}$ given by
\begin{equation}
\pi_{N,k}(\xi):= \frac{1}{Z_{N,k}}\gl^{-A(\xi)}.
\end{equation}
where $\lambda = p/q$, $Z_{N,k}:= \sum_{\xi\in \gO_{N,k}^0} \gl^{-A(\xi)}$,
and 
\begin{equation}\label{def:A}
A(\xi):= \sum_{i=1}^k (N-k+i-\xi_i)\ge 0
\end{equation}
denotes the minimal number of moves that are necessary to go from $\xi$ to the configuration $\xi^{\min}$ where all the particles are 
on the right $\xi^{\min}(x):=\ind_{[N-k+1,N]}(x)$ (this terminology is justified by the fact that $\xi^{\min}$ is minimal for the order introduced in Section \ref{Sec:Prelim}).

\medskip

Recall that the total-variation distance between two probability measures defined on the same state-space $\gO$ is defined by
$$\|\alpha-\gb\|_{TV}=\sup_{A\subset \gO} \alpha(A)-\gb(A),$$
where the $\sup$ is taken over all measurable sets $A$.

\medskip

The mixing time associated to the threshold $\gep\in (0,1)$ is defined by
\begin{equation}
 \Tm^{N,k}(\gep):=\inf\{ t \ge 0 \ : \ d^{N,k}(t)\le \gep \},
\end{equation}
where $d^{N,k}(t)$ denotes the total-variation distance to equilibrium at time $t$ starting from the worst possible initial condition 
\begin{equation}\label{tvdis}
 d^{N,k}(t):= \max_{\xi\in \gO_{N,k}^0} \| P^{N,k}_t(\xi, \cdot)-\pi_{N,k}\|_{TV}.
\end{equation}

We want to study the asymptotic behavior of the mixing time for this system when both the size of the system and the number of particles tend to infinity.
A natural case to consider is when there is a non-trivial density of particles, that is $k/N\to \alpha \in (0,1)$, but we decide to also treat the boundary cases of vanishing density ($\alpha=0$) and full density ($\alpha=1$).
By symmetry we can restrict to the case when $k=k_N\le N/2$: indeed, permuting the roles played by particles and empty sites boils down to reversing the direction of the asymmetry of the jump rates. Note that we will always impose $k\ge 1$ since when $k=0$ the process is trivial.

The asymptotic behavior of $\Tm^{N,k}(\gep)$ in the case of constant bias ($p>1/2$ is fixed when $N$ goes to infinity) has been obtained in a previous work.

\begin{teorema}[Theorem 2 in \cite{LabLac16}]
 
 We have for every $\gep>0$, every $\alpha \in [0,1]$ and every sequence $k_N$ such that $k_N/N \to \alpha$
 
\begin{equation}
\lim_{N\to \infty}\frac{\Tm^{N,k_N}(\gep)}{N}= \frac{(\sqrt{\alpha}+\sqrt{1-\alpha})^2}{p-q}\;.
\end{equation}

\end{teorema}

The result implies in particular that at first order, the mixing time does not depend on $\gep\in(0,1)$, meaning that on 
the appropriate time-scale, for large values of $N$ the distance to equilibrium drops abruptly from $1$ to $0$.
This phenomenon is referred to as \textit{cutoff} and was first observed in the context of card shuffling~\cite{AldDia, DiaSha}. It is known to occur for a large variety of Markov chains, see for instance~\cite{LevPerWil}.
In the context of the exclusion process, it has been proved in \cite{Lac16} that cutoff holds for the Symmetric Simple Exclusion Process (SSEP)
which is obtained by setting $p=1/2$ in the generator \eqref{defgen}.

\begin{teorema}[Theorem 2.4 in \cite{Lac16}]
 When $p=1/2$, for any sequence $k_N$ that tends to infinity and satisfies $k_N\le N/2$ for all $N$, we have
 
  \begin{equation}
   \lim_{N\to \infty}\frac{\Tm^{N,k_N}(\gep)}{N^2\log k_N}= \frac{1}{\pi^2}.
 \end{equation}
 
\end{teorema}

While cutoff occurs in the two cases, it appears to be triggered by different mechanisms. When $p>1/2$,
the mixing time is determined by the time needed for the particle density profile to reach its macroscopic equilibrium:
After rescaling time and space by $N$, the evolution of the particle density has a non-trivial scaling limit (the inviscid Burgers' equation with zero-flux boundary conditions), which fixates at time $\frac{(\sqrt{\alpha}+\sqrt{1-\alpha})^2}{p-q}$.  The first order asymptotic for the mixing time is thus determined by the time 
the density profile needs to reach equilibrium.

\medskip

When $p=1/2$, the right-time scale to observe a macroscopic motion for the particles is $N^2$, and it is worth mentioning that the scaling limit obtained for the particle density 
(the heat equation) does not fixate in finite time. To reach equilibrium, however, we must wait for a longer time, of order $N^2\log N$,
which is the time needed for local fluctuations in the particle density to come to equilibrium.

\medskip

We are interested in studying the process when the drift tends to zero: this requires to understand the transition between these two patterns of relaxation to equilibrium.
Hence we consider $p$ to be a function of $N$ which is such that the bias towards the right $b_N:=p_N-q_N=2p_N-1$ vanishes
\begin{equation}\label{vanishing}
\lim_{N\to \infty} b_N=0.
\end{equation}
In this regime, the model is sometimes called WASEP for Weakly Asymetric Simple Exclusion Process. Its convergence to equilibrium has already been studied in\cite{LabbeKPZ, LevPer16}.  
In \cite{LevPer16} the authors identify 
the order of magnitude of the mixing time as a function of $b_N$ in full generality.
However the approach used in \cite{LevPer16} does not allow to find the exact asymptotic for the mixing time nor to prove cutoff,
and does not answer our question concerning the pattern of relaxation to equilibrium.

\subsection{Results}
\medskip

We identify two main regimes for the pattern of relaxation to equilibrium.
The \textsl{large bias} regime where
\begin{equation}\label{largebias}
\lim_{N\to \infty} \frac{N b_N}{(\log k_N)\vee 1}=\infty.
\end{equation}
and the \textsl{small bias} regime where
\begin{equation}\label{smallbias}
\begin{cases}
\lim\limits_{N\to \infty} \frac{Nb_N}{\log k_N}=0,\\
\lim\limits_{N\to \infty} k_N=\infty. 
\end{cases}
\end{equation}
We identify the asymptotic expression for the mixing time in both regimes.
In the large bias regime we show that the mixing time coincides with the time needed by the particle density
to reach equilibrium like in the constant bias case.

\begin{theorem}\label{Th:largebias}
 
 When \eqref{largebias} holds, and $\lim_{N\to \infty} k_N/N=\alpha\in [0,1]$, we have 
 for every $\gep\in (0,1)$
 \begin{equation}
  \lim_{N\to \infty} \frac{b_N \Tm^{N,k_N}(\gep)}{N}= \left(\sqrt{\alpha}+\sqrt{1-\alpha}\right)^2.
 \end{equation}
 
 \end{theorem}

To state our result in the small bias regime, let us introduce the quantity
\begin{equation}\label{thegap}
\gap_{N}:= (\sqrt{p_N}-\sqrt{q_N})^2+ 4 \sqrt{p_Nq_N} \sin\left( \frac{\pi}{2N}\right)^2,
\end{equation}
which corresponds to the spectral gap associated with the generator \eqref{defgen}. Notice that it does not depend on the number $k$ of particles in the system.
The pattern of relaxation is similar to the one observed in the symmetric case.
\begin{theorem}\label{Th:smallbias}
 
 When \eqref{smallbias} holds, we have  
 
 \begin{equation}
   \lim_{N\to \infty} \frac{\gap_{N} \Tm^{N,k_N}(\gep)}{\log k_N}= \frac12.
   \end{equation}
 
 \end{theorem}

Using Taylor expansion for $\gap_N$ we have, whenever $b_N$ tends to zero
\begin{equation}\label{eq:taylorgap}
\gap_{N}\stackrel{N\to \infty}{\sim} \frac{1}{2}\left( b^2_N+ \left(\frac{\pi}{N}\right)^2 \right).
\end{equation}
Thus in particular we have 
\begin{equation}
\Tm^{N,k_N}(\gep)\sim
\begin{cases}
  \frac{\log k_N}{b_N^2} & \text{if } 1/N \ll b_N\ll \log k_N/N, \\
  \frac{1}{\pi^2} N^2 \log k_N &\text{ if } 0< b_N \ll 1/N, \\
\frac{1}{\pi^2+\beta^2} N^2 \log k_N  &\text{ if } b_N\sim \beta/N.
\end{cases}
\end{equation}

\medskip

Note that our classification of regimes \eqref{largebias}-\eqref{smallbias} does not cover all possible choices of $b_N$.
Two cases have been excluded for very different reasons:
\begin{itemize}
 \item When $b_N=O(N^{-1})$ and $k_N$ is bounded, then we have a system of $k$ diffusive interacting random walks.
 This system does not exhibit cutoff and has a mixing time of order $N^2$ 
 (The upper bound can actually be deduced from argument presented in Section \ref{preuv2} and
 the lower bound is achieved e.g. by looking at the expectation and variance of the number of particles on right half of the segment, like what is done in
 \cite[Section 6]{Morris06} ).
\item When $b_N$ is of order $\log k_N/N$  the time at which the density profile reaches
its equilibrium and the time needed for local fluctuations to reach their equilibrium values are of the same order and 
we believe that there is an interaction between the two phenomena. We provide a more detailed conjecture in Section \ref{conjectos}
\end{itemize}

\begin{remark}
We have not included here results concerning the biased card shuffling considered in \cite{Benjamini, LabLac16}.
Let us mention that while our analysis should also yield optimal bound for the mixing time of this process
when \eqref{largebias} holds (i.e.\ $\Tm^{k,N}(\gep)\sim 2 N /(p_N-q_N)$), it seems much more difficult to  
prove the equivalent of Theorem \ref{Th:smallbias}. The main reason is that the coupling presented in Section \ref{Appendix:Coupling} cannot be 
extended to a coupling on the permutation process. Building on and adapting the techniques presented in \cite[Section 5]{Lac16} it should a priori be possible 
to obtain a result concerning the mixing time starting from an extremal condition (the identity or its symmetric), but this is out of the scope of the present paper.
\end{remark}

\subsection{Conjecture in the regime $b_N\asymp \log k_N/N$}\label{conjectos}

Let us here formulate, and heuristically support a conjecture concerning the mixing-time in the crossover regime 
where 
\begin{equation}\label{limbeta}
\lim_{N\to \infty} \frac{b_N N}{\log k_N}=\gb. 
\end{equation}
for some $\gb\in(0,\infty)$.
For the ease of exposition, while it should be in principle possible to extend the heuristic to the case of positive density (see Remark \ref{pluscomplique} below) 
we restrict ourselves to the case $\lim_{N\to \infty }k_N/N=0$. The justification we provide for the conjecture might be better understood after a first
reading of the entire paper.

\begin{conjecture}\label{jecture}
When $b_N$ and $k_N$ display the asymptotic behavior given by \eqref{limbeta},
we have for every $\gep>0$
\begin{equation}
 \lim_{N \to \infty} \frac{\Tm^{N,k_N}(\gep) \log k_N}{N^2}= \begin{cases}       \frac{2}{\gb}+\frac{1}{\gb^2}, \quad &\text{ if } \gb\le 1/2,\\
       \left( \frac{\sqrt{2}+2\sqrt{\gb}}{2\gb} \right)^2, \quad &\text{ if } \gb\ge 1/2.
                                                              \end{cases}
                                                              \end{equation}

\end{conjecture}

To motivate this conjecture let us first describe the equilibrium measure and its dependence on $\gb$. 
As we are in the low-density regime, the equilibrium measure is quite close to the product measure one would obtain 
for the system without exclusion rules: the $k$ particles are approximately IID distributed with the distance from the right extremity being a  geometric of 
parameter $q_N/p_N\approx e^{- \frac{2\gb \log k}{N}}$.

Hence the probability of having a particle at site $\lfloor zN \rfloor$ for $z\in [0,1]$ is roughly of order  $k^{[1-2\gb (1-z)]+o(1)}/N$.
Thus, while particles are concentrated near the right extremity at equilibrium, the equilibrium ``logarithmic density'' of particles exhibit a non trivial profile 
in the sense that for any $z> 1-(2\gb)^{-1}$ we have
\begin{equation}\label{liquib}
\lim_{\gep\to 0}\lim_{N\to \infty}\frac{\log \left( \sum_{i= (z-\gep)N}^{(z+\gep)N} \xi(i) \right) }{\log k_N} \Rightarrow 1-2\gb (1-z),
\end{equation}
where the convergence holds in probability under the equilibrium measure $\pi_{N,k}$ when $N$ tends to infinity and $\gep$ tends to zero in that order.
The typical distance to zero of the left-most particle at equilibrium is also given by this profile in the sense that 
it is typically $o(N)$ when $\gb\le 1/2$ and of order
$N(1-\frac 1 {2\gb})$ when $\gb\ge 1/2$. While we only give heuristic justification for these statements concerning equilibrium, 
it is worth mentioning that they can be made rigorous by using the techniques exposed in Section 
\ref{Sec:Prelim}.

\medskip

To estimate the mixing time, we assume that the system gets close to equilibrium once the number 
of particles on any ``mesoscopic'' interval  of the form $[ (z-\gep)N, (z+\gep)N ]$ is close to its equilibrium value.
While the mean number of particle is of order  $k^{1-2\gb (1-z)}$ (cf.  \eqref{liquib}), the typical equilibrium fluctuation around this number should be 
the given by the square root due to near-independence of different particles and thus be equal to $k^{\frac{1}{2}-\gb (1-z)}$.

\medskip

To estimate the surplus of particles in this interval at time $t N^2 (\log k)^{-1}$ for $t> 1/\gb$ (note that $t=1/\gb$ is the time of macroscopic equilibrium
where most particles are packed on the right), we consider the number of particles that end up there after keeping a constant drift of order $z (\log k) /(N t)$,
which is smaller than $b_N$.
Neglecting interaction between particles and making a Brownian approximation for the random walk with drift, we obtain that the expected number of particles following this strategy is given by 
$$k \exp\left(- \log k \frac{(\gb t-z)^2}{2 t} \right)=k^{1-\frac{(\gb t-z)^2}{2 t}}.$$

Hence equilibrium should be attained when this becomes negligible with respect to the typical fluctuation $k^{\frac{1}{2}-\gb (1-z)}$ 
for all values of $z$ where we find particles at equilibrium. That is, when the inequality 
\begin{equation}\label{fluqueton}
 1-\frac{(\gb t-z)^2}{2 t}< \frac{1}{2}-\gb (1-z), 
\end{equation}
is valid for all $z\in[0,1]$ if $\gb\le 1/2$ or for all $z\ge 1-\frac{1}{2\gb}$ if $\beta\ge 1/2$.
A rapid computation show that one only needs to satisfy the condition for the smallest value of $z$ (either $1$ or $1-\frac{1}{2\gb}$),
which boils down to finding the roots of a degree two polynomial.
This yields that we must have $t>t_0$ where
\begin{equation}
 t_0:=\begin{cases}
       \frac{2}{\gb}+\frac{1}{\gb^2} &\text{ if } \gb\le 1/2,\\
       \left( \frac{\sqrt{2}+2\sqrt{\gb}}{2\gb} \right)^2 &\text{ if } \gb\ge 1/2.
      \end{cases}
\end{equation}

\begin{remark}\label{pluscomplique}
 Describing the equilibrium ``logarithmic profile'' of particles when the system has positive density is also possible (note that on the right of $(1-\alpha) N$ it is 
 the density of empty-sites that becomes the quantity of interest).  It is thus reasonable to extend the heuristic to that case.
 However the best strategy to produce a surplus of particle in that case becomes more involved, 
 as the zones with positive density of particles, which are described by the hydrodynamic
 evolution given in Proposition \ref{prop:lidro}, play a role in the optimization procedure. 
 For this reason we did not wish to bring the speculation one step further.
\end{remark}

\subsection{Organization of the paper}

In the remainder of the paper we drop the subscript $N$ in $k_N$ in order to simplify the notation.
The article is organised as follows. In Section \ref{Sec:Prelim}, we introduce the representation 
through height functions and collect a few results on the invariant measure, the spectral gap and the hydrodynamic
limit of the process. In Sections \ref{Sec:LBLB} and \ref{Sec:UBLB}, we consider the large bias case and prove respectively 
the lower and upper bounds of Theorem \ref{Th:largebias}: While the lower bound essentially follows from the hydrodynamic
limit, the upper bound is more involved and is one of the main achievement of this paper. In Sections \ref{Sec:LBSB} and \ref{Sec:UBSB}, 
we deal with the small bias case and prove respectively the lower and upper bounds of Theorem \ref{Th:smallbias}. Here again, the lower bound 
is relatively short and follows from similar argument as those presented by Wilson~\cite{Wil04} in the symmetric case, while the upper bound relies on a careful analysis of the area between the processes starting from the maximal and minimal configurations and under some grand coupling.

\section{Preliminaries and technical estimates}\label{Sec:Prelim}

\subsection{Height function ordering and grand coupling}

To any configuration of particles $\xi\in\Omega_{N,k}^0$, we can associate a so-called height function $h=h(\xi)$ defined by $h(\xi)(0) = 0$ and
$$ h(\xi)(x) = \sum_{y=1}^x \big(2\xi(y) - 1\big)\;,\quad x\in \lint 1,N\rint\;.$$
For simplicity, we often abbreviate this in $h(x)$. The height function is a lattice path that increases by $1$ from $\ell-1$ to $\ell$ if there is a particle at site $\ell$, and decreases by $1$ otherwise. Its terminal value therefore only depends on $k$ and $N$. The set of height functions obtained from $\Omega_{N,k}^0$ through the above map is denoted $\Omega_{N,k}$.\\
The particle dynamics can easily be rephrased in terms of height functions: every upward corner ($h(x)=h(x-1)+1=h(x+1)+1$) 
flips into a downward corner ($h(x)=h(x-1)-1=h(x+1)-1$) at rate $p$, while the opposite occurs at rate $q$. 
We denote by $(h^\zeta(t,\cdot),t\ge 0)$ the associated Markov process starting from some initial configuration $\zeta \in \Omega_{N,k}$.\\
It will be convenient to denote by $\wedge$ the maximal height function:
$$ \wedge(x) = x \wedge (2k-x)\;,\quad x\in\lint 1,N\rint\;,$$
and by $\vee$ the minimal height function:
$$ \vee(x) = (-x)\vee(x-2N+2k)\;,\quad x\in\lint 1,N\rint\;.$$
Though the dependence on $k$ is implicit in the notations $\wedge,\vee$, this will never raise any confusion as the value $k$ will be clear from the context.

It is possible to construct simultaneously on a same probability space and in a Markovian fashion, the height function processes $(h_t^\zeta,t\ge 0)$ starting from all initial conditions 
$\zeta\in\cup_k \Omega_{N,k}$ and such that the following monotonicity property is satisfied for all $k$ and all $\zeta,\zeta' \in \Omega_{N,k}$:
\begin{equation}\label{eq:grand}
\zeta \le \zeta' \Rightarrow h_t^\zeta \le h_t^{\zeta'}\;,\quad \forall t\ge 0\;.
\end{equation}
Here, $\zeta\le \zeta'$ simply means $\zeta(x) \le \zeta'(x)$ for all $x\in\lint 0,N\rint$. 
We call such a construction a monotone Markovian grand coupling, 
and we denote by $\P$ the corresponding probability distribution. 
The existence of such a grand coupling is classical, see for instance~\cite[Proposition 4]{LabLac16}.
In a portion of our proof, we require to use a specific grand coupling which is not the one displayed in \cite{LabLac16} and 
for this reason we provide an explicit construction in Appendix \ref{Appendix:Coupling}. 

Once a coupling is specified, by enlarging our probability space, one can also define the process $h^{\pi}_t$ which is started from an initial condition sampled from the equilibrium measure $\pi_{N,k}$, independently of  $h_t^{\zeta}, \zeta\in \gO_{N,k}$.

\medskip

Let us end up this section introducing the (less canonical) notation 
\begin{equation}\label{eq:strict}
 \zeta<\zeta' \quad \Leftrightarrow \quad \left( \zeta\le \zeta' \text{ and } \zeta\ne \zeta' \right).
\end{equation}
We say that a function $f$ on $\Omega_{N,k}$ is increasing (strictly) if $f(\zeta)\le f(\zeta')$ ($f(\zeta)< f(\zeta')$) whenever $\zeta<\zeta'$.
The minimal increment of an increasing function is defined by 
\begin{equation}\label{eq:minincr}
\delta_{\min}(f)=\min_{\zeta,\zeta'\in \gO_{N,k}, \zeta<\zeta'} f(\zeta')-f(\zeta).
\end{equation}

\subsection{The equilibrium measure in the large-bias case}

For $\xi\in \gO_{N,k}^0$ we set 
\begin{equation}\label{def1:lknrkn}
\begin{split}
\ell_{N}(\xi)&=\min\{ x\in \lint 1, N\rint \ : \  \xi(x)= 1 \},\\
r_{N}(\xi)&=\max\{ x\in \lint 1, N\rint  \ : \  \xi(x)= 0 \}.
\end{split}
\end{equation}

A useful observation on the invariant measure is the following. Given $\xi$, we define $\chi(\xi)$ as the sequence of particle spacings:
$$ \chi_i:= \xi_{i+1}-\xi_i\;,\; \text{ for } i\in \lint 1,k-1\rint, \quad  \chi_k=N+1-\xi_k.$$
From \eqref{def:A}, under $\pi_{N,k}$ the probability of a given configuration is proportional to 
\begin{equation}\label{eq:weights}
 \gl^{-\chi_1} \gl^{-2\chi_2} \ldots \gl^{-k \chi_k},
 \end{equation}
In other terms, under the invariant measure the particle spacings $(\chi_i)_{1\le i \le k}$ are distributed like independent geometric variables, with respective parameters $\gl^{-i}$, conditioned to the event $\sum_{i=1}^k \chi_k\le N$.

\begin{lemma}\label{lem:lbeq}
 When \eqref{largebias} holds we have for any $\gep>0$
 \begin{equation}\begin{split}
  \lim_{N\to \infty}& \pi_{N,k}( \ell_N \le (N-k)-\gep N)=0,\\
  \lim_{N\to \infty}& \pi_{N,k}( r_N \ge (N-k)+\gep N)=0,
  \end{split}
 \end{equation}
\end{lemma}

\begin{proof}
By symmetry it is sufficient to prove the result for $\ell_N$ only, but for all $k\in\lint 1,N-1\rint$. Note that there is nothing to prove regarding 
$\ell_N$ if $\alpha:=\lim_{N\to \infty} k/N =1$, so we assume that $\alpha\in[0,1)$.\\
Let $(X_i)_{1\le i \le k}$ be independent geometric variables, with respective parameters $\gl^{-i}$. The sum of their means satisfies 
(recall that $\gl-1$ is of order $b_N$)
\begin{align*}
\sum_{i=1}^k \frac{1}{1-\gl^{-i}} = k+ \sum_{i=1}^k \frac{\gl^{-i}}{1-\gl^{-i}}\le k+ C b_N^{-1} \log \min( k, b_N^{-1})\;,
\end{align*}
for some constant $C>0$. The large bias assumption ensures that $b_N^{-1} \log \min( k_N, b_N^{-1})=o(N)$. Hence using the Markov inequality, we obtain that if $(X_i)_{1\le i \le k}$ is a sequence of such geometric variables, and if
\eqref{largebias} is satisfied, then for any $\gep>0$
\begin{equation*}
\P\Big(\sum_{i=1}^k X_i\ge k+\gep N\Big) \le \frac{\E\big[\big|\sum_{i=1}^k X_i - k\big|\big]}{\gep N}=\frac{\E\big[\sum_{i=1}^k X_i - k\big]}{\gep N}\;,
\end{equation*}
so that
$$ \lim_{N\to\infty}  \P\Big(\sum_{i=1}^k X_i\ge k+\gep N\Big) = 0\;.$$
The above inequality for $\gep<1-\alpha$ implies that $\P(\sum_{i=1}^k X_i\le N)\ge 1/2$ for all $N$ large enough meaning that the conditioning only changes the probability by a factor at most $2$.
Then, we can conclude by noticing that $\ell_N=\xi_1= N-\sum_{i=1}^k \chi_k$.
\end{proof}

\subsection{The equilibrium measure in the small-bias case}\label{Subsec:eqSmallBias}

We aim at showing that with large probability the density of particles everywhere is of order $k^{1+o(1)}/N$.
Given $\xi\in \gO^0_{N,k}$  we let $Q_1(\xi)$, resp. $Q_2(\xi)$, denote the largest gap between two consecutive particles, resp. between two consecutive empty sites.
\begin{equation}
 \begin{split}
Q_1(\xi)&:=\max\{ n\ge 1 \ : \ \exists i\in \lint 0,N-n\rint, \ \forall x\in \lint i+1,i+n\rint, \xi(x)=0 \},\\  
 Q_2(\xi)&:=\max\{ n\ge 1 \ : \ \exists i\in \lint 0,N-n\rint, \ \forall x\in \lint i+1,i+n\rint, \xi(x)=1 \},
 \end{split}
\end{equation}
and $Q(\xi)=\max(Q_1(\xi),Q_2(\xi))$.

\begin{proposition}\label{lem:dens}
 For all $x\in \lint 1,N\rint$, we have
 \begin{equation}\label{encadr}
 \frac{k}{N} \gl^{x-N} \le \pi_{N,k}(\xi(x) = 1) \le \frac{k}{N} \gl^{x-1}\;.
 \end{equation}
Furthermore, there exists a constant $c>0$ such that for all choices of $N\ge 1$, $u>1$ and $p_N\in (1/2,1]$ and all $k\le N/2$ 
\begin{equation}\label{splam}
 \pi_{N,k}\left(Q(\xi)\ge \frac{\gl^{N}  N u}{k} \right) \le  2k e^{-cu}\;.
\end{equation}
\end{proposition}
\begin{proof}
We set $A_x:= \{\xi(x) = 1\}$,
we first prove that for all $y\in \lint 1,N-1\rint$ we have
\begin{equation}\label{swiz}
  \pi_{N,k}(A_y) \le  \pi_{N,k}(A_{y+1}) \le  \gl \pi_{N,k}(A_y).
\end{equation}
We observe that the map $\xi \mapsto \xi^y$ defined in \eqref{flipz} induces a bijection from 
$A_y$ to $A_{y+1}$ and that for every  $\xi\in A^y$, 
\begin{equation}
\pi_{N,k}(\xi) \le \pi_{N,k}(\xi^y)\le \gl \pi_{N,k}(\xi).
\end{equation}
The reader can check indeed that $\pi_{N,k}(\xi^y)= \gl \pi_{N,k}(\xi)$ if  $\xi\in A_y \setminus A_{y+1}$ and that 
 $\xi^y=\xi$ if $\xi\in A_y \cap A_{y+1}$. The desired inequality is then obtained by summing over $\xi\in A_y$.

\medskip

By iterating \eqref{swiz} we obtain 
\begin{equation}
   \gl^{x-N}\pi_{N,k}(A_1) \le  \pi_{N,k}(A_{x}) \le  \gl^{x-1} \pi_{N,k}(A_1).
\end{equation}
By monotony of $\pi_{N,k}(A_y)$ in $y$ and the fact that there are $k$ particles
$$N\pi_{N,k}(A_1) \le \sum_{y=1}^N  \pi_{N,k}(A_y)=k \le N  \pi_{N,k}(A_N),$$ 
and thus \eqref{encadr} can be deduced.

\medskip

We pass to the proof of \eqref{splam}. We can perform the same reasoning as above but limiting ourselves 
to configurations with no particles in some set $I\subset \lint 1, N\rint$.
Setting $B_I:=\{ \forall y\in I, \xi(y)=0 \}$ we obtain similarly to \eqref{swiz} (exchanging directly the content of $x$ and $y$ instead of 
nearest neighbors) that for every  $x,y \in  \lint 1, N\rint  \setminus I$ with $x< y$
\begin{equation}
 \pi_{N,k}(A_x  \cap B_I )  \le \pi_{N,k}(A_y \cap B_I ) \le  \gl^{y-x} \pi_{N,k}(A_x  \cap B_I ).
\end{equation}
This allows to deduce that
\begin{equation}
 \pi_{N,k}(\xi(x)=1 \ | \ \forall y\in I, \xi(y)=0) \ge \frac{k}{N-|I|}\gl^{x-N},
\end{equation}
and yields by induction
 \begin{equation}
   \pi_{N,k}(\forall x\in I, \  \xi(x)=0)\le \left(1-\gl^{-N} \frac{k}{N}\right)^{|I|} \le \exp\left(-|I| \gl^{-N} \frac{k}{N}\right).
\end{equation}
Then noticing that $\{Q_1(\xi)\ge 2m\}$ implies that an interval of the type $\lint mi+1, m(i+1)\rint$ is empty,
a union bound yields that
\begin{equation}
  \pi_{N,k}(Q_1(\xi)\ge 2m)\le \left\lfloor \frac{N}{m}\right\rfloor \exp\left(-m\gl^{-N} \frac{k}{N}\right).
\end{equation}
This remains true for $Q_2(\xi)$ upon replacing $k$ by $N-k$, and this concludes the proof of \eqref{splam} if one choses $m=\frac{\gl^{N}  N u}{2k}$ .
\end{proof}

\subsection{Eigenfunctions and contractions}\label{sec:eigen}

The exact expression of the principal eigenfunction / eigenvalue has been derived in previous works~\cite{LevPer16,LabLac16}. 
It turns out that it can be obtained by applying a discrete Hopf-Cole transform to the generator of our Markov chain. Let us recall some
identities in that direction as they will be needed later on; the details can be found in~\cite[Section 3.3]{LabLac16}. We set
\begin{equation}\label{Eq:rho} \varrho := \big(\sqrt p - \sqrt q\big)^2 \sim \frac{b_N^2}{2}\;,\end{equation}
and we let $a_{N,k}$ be the unique solution of
\begin{equation*}
 \begin{cases}
 	(\sqrt{pq}\, \gD-\varrho)a(x)=0\;,\quad x\in \lint 1,N-1  \rint\;,\\
	 a(0)=1\;,\quad a(N)= \gl^{\frac{2k-N}{2}}  \;,
 \end{cases}
\end{equation*}
where  $\gD$ denotes the discrete Laplace operator
\begin{equation}\label{laplace}
 \gD(f)(x)=f(x+1)+f(x-1)-2f(x), \quad  x\in \lint 1,N-1  \rint.
\end{equation}
If $(h^\zeta_t,t\ge 0)$ denotes the height function process starting from some arbitrary initial condition $\zeta\in\Omega_{N,k}$, then the map
$$ V(t,x) := \E[\lambda^{\frac12 h^{\zeta}_t(x)} - a_{N,k}(x)]\;,\quad t\ge 0\;,\quad x\in \lint 0,N\rint\;,$$
solves
\begin{equation}\label{Eq:V}
\begin{cases}
	\partial_t V(t,x)= (\sqrt{pq}\,\gD- \varrho) V(t,x)\;,\quad x\in\lint 1,N-1  \rint\;.\\
	V(t,0) = V(t,N) = 0\;.
\end{cases}
\end{equation}
This allows to identify $N-1$ eigenvalues and eigenfunctions of the generator $\cL_{N,k}$ of the Markov chain: for every $j\in\{1,\ldots,N-1\}$, the map
\begin{equation}\label{eq:fj}
f^{(j)}_{N,k}(\zeta) = \sum_{x=1}^{N-1} \sin \left(\frac{j x\pi}{N}\right) \left( \frac{\lambda^{\frac12 \zeta(x)} - a_{N,k}(x)}{\lambda - 1} \right)\;,
\end{equation}
defines an eigenfunction with eigenvalue
$$ -\gamma_j = -\varrho - 4\sqrt{p_Nq_N} \sin\left( \frac{j \pi}{2N}\right)^2\;.$$
The eigenvalue $\gamma_1$ corresponds to the spectral gap of the generator (this is related to the fact that the corresponding eigenfunction is monotone,
see \cite[Section 3.3]{LabLac16} for more details),
and for this reason we adopt the notation
$$  \gap_N :=\gamma_1= \varrho + 4\sqrt{p_Nq_N} \sin\left( \frac{\pi}{2N}\right)^2.$$
We also set  $f_{N,k}:=f^{(1)}_{N,k}(\zeta)$ for the corresponding eigenfunction. Notice that this is a \textit{strictly} increasing function 
(recall \eqref{eq:strict}). 
An immediate useful consequence of the eigenvalue equation is that
\begin{equation}\label{eq:contract1}
 \bbE[f_{N,k}(h^{\zeta'}_t)-f_{N,k}(h^\zeta_t)]= e^{-\gap_N t} \left( f_{N,k}(\zeta')-f_{N,k}(\zeta) \right).
\end{equation}

\medskip

To close this section, let us introduce another function which is not an eigenfunction, but is also strictly increasing and enjoys a similar contraction property
$$f^{(0)}_{N,k}(\zeta):=\sum_{x=1}^{N-1} \frac{\gl^{\zeta(x)/2}-a_{N,k}(x)}{\gl-1}.$$
As a direct consequence of \eqref{Eq:V} at time zero, we have (using the notation introduced in \eqref{laplace})
\begin{multline}
 (\cL_{N,k}f^{(0)}_{N,k})(\zeta)= -\varrho f^{(0)}_{N,k}(\zeta)+\sqrt{pq}\sum_{x=1}^{N-1} \frac{\gD(\gl^{\zeta/2}-a_{N,k})(x)}{\gl-1}
 \\=-\varrho f^{(0)}_{N,k}(\zeta)-\frac{\sqrt{pq}}{\gl-1}\left[  \gl^{\frac{\zeta(N-1)}2}+\gl^{\frac{\zeta(1)}{2}}-a_{N,k}(N-1)-a_{N,k}(1)
  \right].
\end{multline}
In particular, we obtain for $\zeta \le \zeta'$
\begin{align*}
 &(\cL_{N,k}f^{(0)}_{N,k})(\zeta')- (\cL_{N,k}f^{(0)}_{N,k})(\zeta)\\&=
 -\varrho \left( f^{(0)}_{N,k}(\zeta')-f^{(0)}_{N,k}(\zeta) \right)- \frac{\sqrt{pq}}{\gl-1} \left[\gl^{\frac{\zeta'(N-1)}2}+\gl^{\frac{\zeta'(1)}2}- \gl^{\frac{\zeta(N-1)}2}-
 \gl^{\frac{\zeta(1)}2}\right]\\
&\le  -\varrho \left( f^{(0)}_{N,k}(\zeta')-f^{(0)}_{N,k}(\zeta) \right).
\end{align*}
Considering a monotone coupling between $(h^{\zeta'}_t)_{t\ge 0}$ and $(h^{\zeta}_t)_{t\ge 0}$, we obtain that 
\begin{multline}
 \partial_t  \bbE[f^{(0)}_{N,k}(h^{\zeta'}_t)-f^{(0)}_{N,k}(h^\zeta_t)]=
 \bbE[(\cL_{N,k}f^{(0)}_{N,k})(h_t^{\zeta'})- (\cL_{N,k}f^{(0)}_{N,k})(h_t^\zeta)]\\ \le   -\varrho \bbE[f^{(0)}_{N,k}(h^{\zeta'}_t)-f^{(0)}_{N,k}(h^\zeta_t)],
\end{multline}
and thus
\begin{equation}\label{eq:contract0}
 \bbE[f^{(0)}_{N,k}(h^{\zeta'}_t)-f^{(0)}_{N,k}(h^\zeta_t)]\le e^{-\varrho t} \left( f^{(0)}_{N,k}(\zeta')-f^{(0)}_{N,k}(\zeta) \right).
\end{equation}

\subsection{The hydrodynamic limit}

 We are interested in the macroscopic evolution of the height function. 

For $\alpha\in [0,1]$, we define $\vee_{\alpha} : [0,1] \to \bbR$, $\wedge_{\alpha} : [0,1] \to \bbR$ as 
\begin{equation*}
\vee_{\alpha}(x):= \max(-x,x-2(1-\alpha))\;,\qquad \wedge_{\alpha}(x):= \min(x,2\alpha-x)\;,
\end{equation*}
and we let $g_{\alpha}: \bbR_+\times [0,1]\to \bbR$ be defined as follows
\begin{equation*}\begin{split}
g^0_{\alpha}(t,x)&:= \begin{cases} \alpha-\frac{t}{2}-\frac{(x-\alpha)^2}{2t}, \quad &\text{ if } |x-\alpha| \le t,  \\
\wedge^{\alpha}(x), \quad & \text{ if } |x-\alpha| \ge t,
\end{cases}\\
g_{\alpha}(t,x)&:= \max(\vee_{\alpha}(x), g^0_{\alpha}(t,x)).
\end{split}
\end{equation*}

\begin{proposition}\label{prop:lidro}
Assume that $Nb_N = N(p_N-q_N) \to\infty$ and that $k_N/N \to \alpha \in (0,1)$. 
Then, after an appropriate space-time scaling,  $h^{\wedge}(\cdot,\cdot)$ converges to $g_\alpha$ 
in probability as $N\to\infty$. More precisely we have for any $\gep>0$, $T>0$,
\begin{equation}
 \lim_{N\to \infty} \bbP\left[ \sup_{t\le T}\sup_{x\in[0,1]}\left| \frac{1}{N} h\Big( \frac{Nt}{b_N}, Nx \Big)-g_{\alpha}(t,x) \right|\ge \gep \right]=0.
\end{equation}
\end{proposition}
\begin{proof}
This is essentially the content of~\cite[Th 1.3]{LabbeKPZ} where the hydrodynamic limit of the density of particles 
is shown to be given by the inviscid Burgers' equation with zero-flux boundary conditions: when starting from the maximal initial condition, 
this yields (after integrating the density in space) the explicit solution $g_\alpha$.\\
Actually the setting of~\cite[Th 1.3]{LabbeKPZ} is more restrictive as the number of particles is taken to be
$k=N/2$ and $p_N-q_N = 1/N^{\alpha}$ with $\alpha \in (0,1)$. However, a careful inspection of the proof shows that we only require $N^{1-\alpha}$
to go to infinity: this corresponds to the assumption $N(p_N-q_N) \to \infty$ which is in force in the statement of the proposition 
so that the proof carries through \textit{mutatis mutandis}.

\end{proof}

\section{Lower bound on the mixing time for large biases}\label{Sec:LBLB}

In the large bias case, the last observable that equilibrates is the position of the leftmost particle.
Obtaining a lower bound on the mixing time is thus relatively simple: we have to show that for arbitrary $\delta>0$ at time 
$$s_\gd(N):=[(\sqrt{\alpha}+\sqrt{1-\alpha})^2-\delta]N b_N^{-1},$$ the leftmost particle has not reached its equilibrium position given by 
 Lemma \ref{lem:lbeq}. This is achieved by using the hydrodynamic limit for $\alpha>0$, and a simple comparison argument for $\alpha=0$.

 \begin{proposition}
 When \eqref{largebias} is satisfied,
  for every $\delta>0$ we have 
  \begin{equation}
 \lim_{N\to \infty} \left\| \bbP\left( \ell_N(\eta^{\wedge}_{s_{\delta}(N)}) \in \cdot \right)- \pi_N(\ell_N \in \cdot)\right\|_{TV}=1.
  \end{equation}
As a consequence for all $\gep>0$ and $N$ sufficiently large
$$\Tm^{N,k}(1-\gep)\ge s_{\delta}(N).$$ 
 \end{proposition}

\subsection{The case $\alpha=0$}
Given $\gd>0$ we want to prove that the system is not mixed at time $s_\gd(N):=(1-\gd)N b_N^{-1}$.
We know from Lemma \ref{lem:lbeq}, that when $\alpha=0$ and \eqref{largebias} holds, at equilibrium we have
$$\lim_{N\to \infty}\pi_{N,k}(\ell_N\le (1-\gd/2)N)=0.$$
On the other hand  
observing that the position of the first particle is dominated by a random walk on $\bbN$ with 
jump rates $p_N$ to the right and $q_N$ to the left,
it is standard to check that whenever $\lim_{N\to \infty} b_N N=\infty$ 
\begin{equation}
  \lim_{N\to \infty} \bbP\left( \ell_N(\eta^{\wedge}_{s_\delta})\le (1-\gd/2)N \right)=1.
\end{equation}
\qed

\subsection{The case $\alpha\in (0,1/2]$}

Setting $x_\delta:= 1-\alpha-c_\alpha \delta$, for some positive constant $c_\alpha$ sufficiently small, we observe that
the hydrodynamic profile at the rescaled time corresponding to $s_\delta$  is above the minimum at $x_{\delta}$
$$g_{\alpha}(x_{\delta}, (\sqrt{\alpha}+\sqrt{1-\alpha})^2-\delta)>\vee_{\alpha}(x_{\delta}).$$
The reader can check that $c_{\alpha}=1/3$ works for all $\alpha\in(0,1/2)$.
Thus whenever $\lim_{N\to \infty} b_N N=\infty$
Proposition \ref{prop:lidro} yields that 
\begin{equation}
\lim_{N\to \infty} \bbP\left(  \ell_N(\eta^{\wedge}_{s_\delta}) \le  (1-\alpha-c_{\alpha}\delta)N \right)=1.
\end{equation}
On the other hand we know from Lemma \ref{lem:lbeq} that when \eqref{largebias} holds, at equilibrium we have for any $\delta>0$
$$\lim_{N\to \infty}\pi_{N,k}(\ell_N\le (1-\alpha-c_{\alpha}\delta)N)=0.$$

\qed

\section{Upper bound on the mixing time for large biases}\label{Sec:UBLB}

Let us recall how a grand coupling satisfying the order preservation property \eqref{eq:grand} is of help to establish an upper bound on the mixing time.
Recalling \eqref{tvdis}, we have by the triangle inequality
\begin{equation}\label{eq:zone}
d^{N,k}(t)\le  \max_{\zeta,\zeta' \in \gO_{N,k}}\|P^{N,k}_{t}(\zeta',\cdot)-P^{N,k}_{t}(\zeta,\cdot)\|_{TV}.
\end{equation}
On the other hand if $\bbP$ is a monotone grand coupling, one observes that the extremal initial conditions are the last to couple so that one has 
\begin{equation}\label{eq:ztwo}
 \|P^{N,k}_{t}(\zeta',\cdot)-P^{N,k}_{t}(\zeta,\cdot)\|_{TV}\le \bbP[ h^{\zeta}_t\ne h^{\zeta'}_t] \le \bbP[ h^{\vee}_t\ne h^{\wedge}_t].
\end{equation}
Hence to establish an upper bound on the mixing time, it suffices to obtain a good control on the merging time 
\begin{equation}\label{eq:mergin}
 \tau:=\inf\{ t>0 \ : \ h^{\vee}_t= h^{\wedge}_t\}.
\end{equation}
Let us set for this section 
\begin{equation}\label{deltaz}
t_\gd(N):=[(\sqrt{\alpha}+\sqrt{1-\alpha})^2+\delta]N b_N^{-1}.
\end{equation}
\begin{proposition}\label{lapropo}
When \eqref{largebias} is satisfied,
for every $\delta>0$, and any monotone grand coupling we have 
\begin{equation}
 \lim_{N\to \infty} \bbP\left( \tau \le t_{\delta}(N)\right)=1.
\end{equation}
As a consequence for all $\gep>0$ and $N$ sufficiently large $\Tm^{N,k}(\gep)\le t_{\delta}(N)$.
 \end{proposition}

When $\alpha=1/2$, a sharp estimate on $\tau$ can be obtained directly from spectral considerations (Section \ref{specz}), but when $\alpha\in [0,1/2)$
we need a refinement of the strategy used in \cite{LabLac16}: The first step (Proposition \ref{lerimo}) is to 
obtain a control on the position of the leftmost particle which matches the lower bound provided 
by the hydrodynamic limit. This requires a new proof since the argument used in \cite{LabLac16} is not sharp enough to cover all biases.
The second step is to use contractive functions once the system is at macroscopic equilibrium, this is sufficient to treat most cases.
A third and new step is required to treat the case when the  bias $b_N$ of order $\log N/ N$ or smaller: as we are working under the assumption 
\eqref{largebias} we only need to treat this case when $k=N^{o(1)}$. In this third step we use diffusive estimates to control the hitting time of zero for the function
$f^{(0)}_{N,k}(h^{\wedge}_t)-f^{(0)}_{N,k}(h^{\vee}_t)$ where $f^{(0)}$ was introduced in Subsection \ref{sec:eigen}.

\subsection{The special case $\alpha=1/2$}\label{specz}

In the special case $\alpha=1/2$, a sharp upper-bound can be deduced in a rather direct fashion from spectral estimates repeating the computation performed 
in \cite[Section 3.2]{Wil04} for the symmetric case. This fact is itself a bit surprising since this method does not yield the correct upper bound in the symetric case
nor in the constant bias case. 

\medskip

Recall the definition of $f_{N,k}$ in Equation \eqref{eq:fj} and below. It being a strictly monotone function and $\bbP$ being a monotone coupling,
we obtain using Markov's inequality (recall \eqref{eq:minincr})
\begin{equation}\label{eq:mark}
 \bbP(\tau>t)= 
 \bbP\left[ f_{N,k}(h^{\wedge}_t) > f_{N,k}(h^{\vee}_t)\right] \le
 \frac{\bbE\left[ f_{N,k}(h^{\wedge}_t) - f_{N,k}(h^{\vee}_t)\right]}{\delta_{\min}(f_{N,k})}.
 \end{equation}
The expectation decays exponentially with rate $\gap_N$ \eqref{eq:contract1} and it is not difficult to check that 
 \begin{equation}\label{delmin}
\delta_{\min}(f_{N,k})\ge  \gl^{(k-N)/2} \sin\left(\frac{\pi}{N}\right)\ge N^{-1}\gl^{(k-N)/2}.
\end{equation}
Hence Equation \eqref{eq:mark} becomes
\begin{equation}\label{eq:zthree}
  \bbP(\tau>t)
  \le  N \gl^{\frac{N-k}{2}} e^{-\gap_N t}\left( f_{N,k}(\wedge)-f_{N,k}(\vee) \right)
  \le  \frac{N^2 \gl^{N/2}}{\gl-1}e^{-\gap_N t},
\end{equation}
where in the last inequality we used that 
$$\sin\left(\frac{x\pi}{N}\right)\left(\gl^{\wedge(x)/2}-\gl^{\vee(x)/2}\right)\le \gl^{\wedge(x)/2}\le \gl^{k/2}.$$

\medskip

Recall that we assume that \eqref{largebias} holds and $b_N$ tends to zero.
Recalling  \eqref{eq:taylorgap}, we obtain the following asymptotic equivalent 
\begin{equation}\label{eq:lezequiv}
\gap_N \stackrel{N\to \infty}{\sim} b_N^2/2 \quad \text{ and } \quad  \log \gl_N \stackrel{N\to \infty}{\sim} 2b_N
\end{equation}
Furthermore for $N$ sufficiently large we have 
$(\gl-1)^{-1}\le N$.
Hence recalling that 
 $t_\delta=(2+\delta)b^{-1}_N N$ and using \eqref{eq:lezequiv} in \eqref{eq:zthree} we obtain for all $N$ sufficiently large
\begin{equation}
 \bbP(\tau>t_\delta)\le 
  \frac{N^2}{\gl-1} \exp\left(\frac{N}{2}\log \gl -\gap_N t_{\gd} \right)
 \le N^3 e^{-\frac{\delta}{4}b_N N}.
\end{equation}
and the left-hand side vanishes when $N$ tends to infinity as a consequence of \eqref{largebias} (recall that as $\alpha=1/2$, $k$ is of order $N$)
\qed

\begin{remark}
The reason why the computation above yields a sharp upper bound is because: (A) The difference of order between
$\delta_{\min}(f_{N,k})$ and the typical fluctuation of $f_{N,k}$ at equilibrium is negligible in the computation. 
(B) Until shortly before the mixing time the quantity $\log \left[f_{N,k}(h^{\wedge}_t) - f_{N,k}(h^{\vee}_t)\right]$ 
has the same order of magnitude as  $\log \bbE\left[f_{N,k}(h^{\wedge}_t) - f_{N,k}(h^{\vee}_t)\right].$
In the case of symmetric exclusion, $(A)$ does not hold, while when the bias is constant, (B) fails to hold.
In the weakly asymmetric case, when $\alpha\ne 1/2$, the reader can check by combining Propositions \ref{prop:lidro}
and \ref{lerimo}
that (B) holds until time $4\alpha$ (in the macroscopic time-scale) after which $g_{\alpha}(t,\cdot)$ 
stops to display a local maximum in the interval $(l_{\alpha}(t),r_{\alpha}(t))$ and $f_{N,k}(h^{\wedge}_t) - f_{N,k}(h^{\vee}_t)$ starts to decay much faster than its average.
\end{remark}

\subsection{The case $\alpha \ne1/2$: scaling limit for the boundary processes}

In order to obtain a sharp upper-bound for $\alpha \ne 1/2$, we rely on a scaling limit result in order to control the value of 
$f_{N,k}(h^{\wedge}_t)-f_{N,k}(h^{\vee}_t)$ up to a time close to the mixing time, and then we use the contractive estimate \eqref{eq:contract1} to couple $h^{\wedge}_t$ with $h^{\vee}_t$. Note that 
Proposition \ref{prop:lidro} is not sufficient to estimate $f_{N,k}(h^{\wedge}_t)$: we also need a control on the positions  
of the left-most particle and right-most empty site in our particle configuration.

In the case when $\alpha=0$ and the bias is of order $\log N/N$ or smaller (this is possible when \eqref{largebias} 
is satisfied and $k$ grows slower than any power of $N$), we need an additional step, based on diffusion estimates, to couple the two processes.
In this last case also, the factor $N^{-1}$ in \eqref{delmin} causes some difficulty. For that reason
we use $f^{(0)}_{N,k}$ and \eqref{eq:contract0} instead of $f_{N,k}$ and \eqref{eq:contract1}: observe that
$\delta_{\min}(f^{(0)}_{N,k})=\gl^{\frac{k-N}{2}}$.

\medskip

\noindent Let us define $[L_N(t),R_N(t)]$ to be the interval on which $h^{\wedge}_t$ and $\vee$ differ. More explicitly, we set
\begin{equation}\begin{split}\label{def:lknrkn}
L_N(t)&:=\max\{x \ : \ h^{\wedge}(t,x)=-x\},\\
R_N(t)&:=\min\{x \ : \ h^{\wedge}(t,x)= x-2(N-k)\},
\end{split}
\end{equation}
or equivalently $L_N(t):=\ell_N(\eta^{\wedge}_t)-1$ and  $R_N(t):=r_N(\eta^{\wedge}_t)$.

We let $\ell_\alpha$ and $r_{\alpha}$ denote the most likely candidates for the scaling limits of $L_N$ and $R_N$ that can be inferred 
from the hydrodynamic behavior of the system (cf. Proposition \ref{prop:lidro}):
\begin{equation*}\begin{split}
\ell_\alpha(t)&=\begin{cases} 0 \quad &\text{ if } t\le \alpha\;,\\ 
                      (\sqrt{t}-\sqrt{\alpha})^2 \quad &\text{ if } t\in \left(\alpha, (\sqrt{\alpha}+\sqrt{1-\alpha})^2 \right)\;,\\
                      1-\alpha \quad &\text{ if } t\ge   (\sqrt{\alpha}+\sqrt{1-\alpha})^2\;,
\end{cases}
 \\
r_{\alpha}(t)&=\begin{cases} 1 \quad &\text{ if } t\le 1-\alpha\;,\\ 
                      1-(\sqrt{t}-\sqrt{1-\alpha})^2 \quad &\text{ if } t\in \left(1-\alpha, (\sqrt{\alpha}+\sqrt{1-\alpha})^2 \right)\;,\\
                      1-\alpha \quad &\text{ if } t\ge   (\sqrt{\alpha}+\sqrt{1-\alpha})^2\;.
\end{cases}\end{split}
 \end{equation*}
We prove that $\ell_{\alpha}$ and $r_{\alpha}$ are indeed the scaling limits of $L_N$ and $R_N$.
\begin{proposition}\label{lerimo}
 If \eqref{largebias} holds and $k_N/N\to \alpha$ then for every $t>0$ we have the following convergences in probability
   \begin{equation}
  \lim_{N\to \infty}  \frac{1}{N} L_N\left( b^{-1}_NNt  \right)= \ell_\alpha(t),\qquad\lim_{N\to \infty}  \frac{1}{N} R_N\left( b_N^{-1} Nt  \right)= r_\alpha(t).
  \end{equation}

\end{proposition}
\begin{remark}
 The assumption \eqref{largebias} is optimal for the above result to hold. To see this, the reader can check that  when  \eqref{largebias} fails, 
at equilibrium $\ell_N$ and $r_N$ are typically at a macroscopic distance from $(1-\alpha)N$.
 \end{remark}

The proof of Proposition \ref{lerimo} is presented in the next subsections. Let us now check that it yields the right bound on mixing time. First, notice that the inequalities \eqref{eq:mark} still hold with $f_{N,k}$ replaced by $f_{N,k}^{(0)}$ since the latter is also a strictly increasing function in the sense of \eqref{eq:strict}. Next observe that Proposition \ref{lerimo} allows an acute control on the quantity
$$ \frac{f^{(0)}_{N,k}(h^{\wedge}_t)-f^{(0)}_{N,k}(h^{\vee}_t)}{\delta_{\min} (f^{(0)}_{N,k})}\;.$$
We summarize the argument in a lemma.

\begin{lemma}\label{lem:difrence}
 Set $D_N(\zeta):= \max\left(|L_N(\zeta)-N+k|, |R_N(\zeta)-N+k|\right)$. We have for $\zeta' \ge \zeta$
 \begin{equation}
 \frac{f^{(0)}_{N,k}(\zeta')-f^{(0)}_{N,k}(\zeta)}{\delta_{\min} (f^{(0)}_{N,k})}\le Nk \gl^{D_N(\zeta')}.
 \end{equation}
 \end{lemma}

 \begin{proof}
We assume that $\zeta' \ne \zeta$. Then,
  \begin{equation}\label{linex}
\frac{\lambda^{\frac12 \zeta'(x)} - \lambda^{\frac12 \zeta(x)}}{\gl-1} = \sum_{n=0}^{\frac{\zeta'(x)-\zeta(x)}{2}-1} \lambda^{\frac12\zeta(x)+n}\\
\le \lambda^{\frac{\zeta'(x)}{2}} \frac{\left(\zeta'(x)-\zeta(x)\right)}{2}.
\end{equation}
Now for $x\le L_N(\zeta')$ or $x\ge R_N(\zeta')$ we necessarily have  $\zeta(x)=\zeta'(x)=\vee(x)$.
For $x\in \lint  L_N(\zeta')+1,R_N(\zeta')-1\rint$, the fact that $\zeta'$ is $1$-Lipschitz yields 
\begin{equation}
 \zeta'(x)\le k-N+2 D_N(\zeta').
\end{equation}
Recall that $\delta_{\min} (f^{(0)}_{N,k}) = \gl^{(k-N)/2}$. Hence one obtains from \eqref{linex}
\begin{align*}
\frac{f^{(0)}_{N,k}(\zeta')-f^{(0)}_{N,k}(\zeta)}{\delta_{\min} (f^{(0)}_{N,k})}
&\le 
  \sum_{x=1}^N \lambda^{\frac{\zeta'(x)-(k-N)}{2}} \frac{\left(\zeta'(x)-\zeta(x)\right)}{2} \\ &\le 
  \lambda^{D_N(\zeta')} \sum_{x=1}^N \frac{\left(\zeta'(x)-\zeta(x)\right)}{2}\le   \lambda^{D_N(\zeta')} Nk.
\end{align*}
In the last inequality we simply used that  $\zeta'(x)-\zeta(x)\le 2k$ (there are at most $k$ sites where the increment of $\zeta'$
is larger than that of $\zeta$).

 \end{proof}

 We can now apply Proposition \ref{lerimo} to obtain an estimate on the mixing time.
 For convenience we treat the case of smaller bias separately.

\subsubsection{Proof of Proposition \ref{lapropo} when $b_N\gg (\log N)/ N$ }\label{preuv1}
We assume that
\begin{equation}\label{hypo}
\lim_{N\to \infty}(b_N N)/\log N=\infty.
\end{equation}
We consider first the system at time $t_0(N):= (\sqrt{\alpha}+\sqrt{1-\alpha})^2 N b^{-1}_N$.
From Proposition \ref{lerimo}, we know that at time $t_0$, $L_N$ and $R_N$ are close to their equilibrium positions:
we have for $N$ sufficiently large and arbitrary $\delta, \gep>0$
\begin{multline}\label{eq:lax}
 \bbP[  L_N(t_0)\ge N-k-(\delta/20) N \ ; \  R_N(t_0)\le N-k+(\delta/20) N  ]\\
 =: \bbP[\cA_N] \ge 1-(\gep/2).
\end{multline}
We let $\cF_t$ denote the canonical filtration associated with the process.
For $t\ge t_0$,  repeating \eqref{eq:mark} starting at time $t_0$ for $f^{(0)}_{N,k}$ 
and combining it with \eqref{eq:contract0}, we obtain that 
\begin{equation}\label{eq:lux}
    \bbP[ \tau>t \ | \ \cF_{t_0}]\le e^{-\varrho(t-t_0)}\frac{ f^{(0)}_{N,k}(h^{\wedge}_{t_0})-f^{(0)}_{N,k}(h^{\vee}_{t_0})}{\delta_{\min}(f^{(0)}_{N,k})}\;.
\end{equation}

Note that on the event $\cA_N$, we have $D_N(h^{\wedge}_{t_0})\le \delta N/20$.
Thus using Lemma \ref{lem:difrence} to bound the r.h.s.\ we obtain 
\begin{equation}\label{eq:lox}
 \bbE\left[\frac{ f^{(0)}_{N,k}(h^{\wedge}_{t_0})-f^{(0)}_{N,k}(h^{\vee}_{t_0})}{\delta_{\min}(f^{(0)}_{N,k})} \ \Big| \ \cA_N\right] \le kN\gl^{\frac{\delta N}{20}}.
\end{equation}
Hence averaging \eqref{eq:lux} on the event $\cA_N$ one obtains
we obtain
\begin{equation}
      \bbP( \tau>t )\le   \gep/2+  \bbP[ \tau>t \ | \ \cA_N]\le  \gep/2 + e^{-\varrho(t-t_0)}kN\gl^{\frac{\delta N}{20}}.
\end{equation}
For $t=t_\delta=t_0+  \delta b_N^{-1} N$,
replacing $\varrho$ and $\log \gl$ by their equivalents given in \eqref{Eq:rho} and \eqref{eq:lezequiv}, one can check that for $N$ sufficiently large one has
  \begin{equation}\label{oazt}
 N k \gl^{\frac{\delta N}{20}} e^{-\varrho (t_{\delta}-t_0)} \le N k  e^{-\frac{\delta N b_N}{20}}\le \gep/2,
  \end{equation}
  where the last inequality is valid for $N$ sufficiently large provided that \eqref{hypo} holds.

\qed

\subsubsection{Proof of Proposition \ref{lapropo}: the general case}\label{preuv2}

If we no longer assume that \eqref{hypo} holds, then an additional step is needed in order to conclude: this step relies on diffusion estimates proved in Appendix \ref{sec:diffu}.
From \eqref{eq:lox} and \eqref{eq:contract0}, for any $\gep, \gd >0$ we have  for $N$ sufficiently large (recall \eqref{deltaz})
\begin{multline}
 \bbE\left[\frac{ f^{(0)}_{N,k}(h^\wedge_{t_{\delta/2}}) - f^{(0)}_{N,k}(h^\vee_{t_{\delta/2}})}{\delta_{\min}(f^{(0)}_{N,k})} \ \Big| \ \cA_N \right]\\ \le   
 e^{-\varrho(t_{\delta/2}-t_0)}\gl^{ \frac{\delta N}{20}} k N \le
e^{-\frac{\delta N b_N}{40}} k N\le e^{-\frac{\delta N b_N}{50}} N,
\end{multline}
where the second inequality relies on the the asymptotic equivalence in \eqref{eq:lezequiv} and the last one on \eqref{largebias}.\\
Now we can conclude using Proposition \ref{prop:solskjaer}-(i) with $a:= 4\gep^{-1} e^{-\delta N b_N/50} N$ and
$$ (M_{s})_{s\ge 0}:= \left(\frac{f^{(0)}_{N,k}(h^\wedge_{t_{\delta/2}+s})-f^{(0)}_{N,k}(h^\vee_{t_{\delta/2}+s})}{\delta_{\min}(f^{(0)}_{N,k})}\right)_{s\ge 0}.$$
Indeed $M_{s}$ is a non-negative supermartingale whose jumps are of size at least $1$ (recall that we have divided the weighted area by $\delta_{\min}(f^{(0)}_{N,k})$ in the definition of $M$). Furthermore, up to the merging time $\tau$, the two interfaces $h^\wedge$ and $h_\vee$ differ on some interval: on this interval $h^\wedge$ makes an upward corner ($\Delta h^\wedge < 0$) and $h^\vee$ makes a downward corner ($\Delta h^\vee > 0$). Consequently, the jump rate of $M$ is at least $1$ up to its hitting time of $0$. From Markov's inequality we have (recall \eqref{eq:lax}) 
\begin{equation}
\bbP[M_0 > a]\le \bbP[\cA^{\cc}_N]+ a^{-1} \bbE[M_0 \ | \ \cA_N ] \le 3\gep/4.
\end{equation}
Setting $r_{\delta}:=  (\delta/2) N b^{-1}_N$ and applying \eqref{lopes}, we have for all $N$ sufficiently large
\begin{equation}
\bbP[M_{r_{\delta}}>0 \ | \ M_0\le a ]\le 4 a (r_\delta)^{-1/2}\le  \frac{16 \gep^{-1}}{\sqrt{\delta/2}} \sqrt{N b_N} e^{-\frac{\delta N b_N}{50}}\le \gep/4\;,
\end{equation}
where the last inequality comes from the fact that $Nb_N$ diverges. Hence we conclude by observing that for $N$ sufficiently large
\begin{equation}
\bbP[\tau>t_{\delta}] =\bbP[M_{r_{\delta}}>0] \le \bbP[M_0 > a]+\bbP[M_{r_{\delta}}>0 \ | \ M_0\le a ]\le \gep.
 \end{equation}
 \qed

\subsection{An auxiliary model to control the speed of the right-most particle}\label{sec:auxi}
 
Our strategy to prove Proposition \ref{lerimo} is to compare our particle system with another one on the infinite line, 
for which a stationary probability exists.
We consider $n$ particles performing the exclusion process on the infinite line with jump rate $p$ and $q$ 
and we add a ``slower'' $n+1$-th particle on the right to enforce existence of a stationary probability for the particle spacings.
To make the system more tractable this extra particle is only allowed to jump to the right (so that it does not feel the influence of the $n$ others).
Note that in our application, the number of particles $n$ does not necessarily coincide with $k$.\\
The techniques developed in this section present some similarities to those used for the constant bias case in
\cite[Section 6]{LabLac16}, but also present several improvements, the main conceptual change being the addition of a slow particle instead of modifying the biases
in the process.
This novelty presents two advantages: Firstly it considerably simplifies the computation
since martingale concentration estimates are not needed any more. Secondly this allows to obtain control for the whole large bias regime \eqref{largebias},
something that cannot be achieved even by optimizing all the parameters involved in \cite[Section 6]{LabLac16}.

 \medskip

\noindent More formally we consider a Markov process $(\hat \eta(t))_{t\ge 0}$ on the state space
$$\Theta_n:= \{ \xi\in \bbZ^{n+1}  \ : \ \xi_1<\xi_2<\dots<\xi_{n+1}\}.$$
The coordinate $\hat \eta_i(t)$ denotes the position of the $i$-th leftmost particle at time $t$.
The dynamics is defined as follows: the first $n$ particles, $\hat \eta_{i}(t)$, $i\in \lint 1,n\rint$
perform an exclusion dynamics with jump rates $p$ to the right and $q$ to the left while the last one 
$\hat \eta_{n+1}(t)$ can only jump to the right and does so with rate $\gb b=\gb(p-q)$, for some $\beta<1$.

We assume furthermore that initially we have $\hat \eta_{n+1}(0)=0$.
The initial position of the other particles is chosen to be random in the following manner.
We define 
\begin{equation}\label{ladef}
 \mu_i:=\gb+\gl^{-i}(1-\gb).
\end{equation}
and we assume that the spacings $\left( \hat \eta_{i+1}(0)-\hat \eta_{i}(0) \right)_{i=1}^n$ 
are independent with Geometric distribution
\begin{equation}\label{geom}
\bbP \left[ \hat \eta_{i+1}(0)-\hat \eta_i(0)= m \right]=(1-\mu_i)\mu_i^{m-1}\;,\quad m\ge 1\;.
\end{equation}

Our aim is to prove the following control on the position of the first particle in this system, uniformly in $\gb$ and $n$.
In Subsections \ref{Subsec:alpha0} and \ref{Subsec:alphaNon0}, we use this result in order to control the position of $L_N(t)$. 

\begin{proposition}\label{deviats}
We have, 
 \begin{multline}\label{eq:woof}
\lim_{A\to \infty}\sup_{t\ge 1, n\in \bbN, \gb\in(0,1)} 
\bbP\bigg[\hat \eta_1(t) \le t\gb b\\
- A \left( \sqrt{bt}+ \frac{1}{1-\gb}\left[n+ b^{-1} \log \min(n,b^{-1}) \right] \right) \bigg]= 0.
 \end{multline}
\end{proposition}

The statement is not hard to prove, the key point is to observe that the distribution of particle spacings is 
stationary.

 \begin{lemma}\label{station}
For all $t\ge 0$, $\left( \hat \eta_{i+1}(t)-\hat \eta_{i}(t) \right)_{i=1}^n$ are independent r.v.~with distribution 
given by \eqref{geom}.
\end{lemma}

\begin{proof}
We use the notation $(m_i)_{i=1}^n\in \bbN^n$ to denote a generic element in the configuration space for the process $\left( \hat \eta_{i+1}(t)-\hat \eta_{i}(t) \right)_{i=1}^n$. We need to show that the measure defined above is stationary.

A measure $\pi$ is stationary if and only if we have

\begin{align*}
& p \pi(m_1+1,\dots,m_n)   \\
+\;&\sum_{i=1}^{n-1}\left[p\pi(\dots, m_i-1,m_{i+1}+1,\dots)+q\pi(\dots,m_{i-1}+1,m_i-1,\dots)\right]\ind_{\{m_i\ge 2\}}\\
+\;&q\pi(\dots,m_{n-1}+1,m_n-1) \ind_{\{m_n\ge 2\}}  +\gb b \pi(m_1,\dots,m_n-1)\ind_{\{m_n\ge 2\}}\\
&=
 \pi(m_1,\dots,m_n)\left(q + \sum_{i=1}^{n-1} (p+q) \ind_{\{m_i\ge 2\}}+ p\ind_{\{m_n\ge 2\}}  + \gb b \right),
\end{align*}
where in the sums, the dots stand for coordinates that are not modified (and $m_{i-1}$ simply has to be ignored when $i=1$).
If we assume that $\pi$ is the product of geometric laws with respective parameters $\mu_i$ (not yet fixed) then the equation above is equivalent to the system
\begin{equation}\label{dasistem}
\begin{cases}
 p\mu_1=q+ \gb(p-q), &\\
q \frac{\mu_{i-1}}{\mu_i}+p \frac{\mu_{i+1}}{\mu_i}=p+q, &\quad \forall i\in \lint 1, n-1 \rint,\\
q\frac{\mu_{n-1}}{\mu_n}+\frac{\gb(p-q)}{\mu_n}=p.&
\end{cases}
\end{equation}
where we have taken the convention $\mu_0=1$.
One can readily check that $\mu_i$ given by \eqref{ladef} satisfies this equation.
\end{proof}

\begin{remark}
Note that the equations \eqref{dasistem} can be obtained directly simply by using the fact that the expected drifts of the particles starting from the geometric distributions are given by
$p\mu_i-q \mu_{i-1}$ for the $i$-th particle $i\in \lint 1,n \rint$ and  $\gb(p-q)$ for the $n+1$-th particle, and that stationarity implies that the drifts are all equal. However, the proof is necessary to show that this condition is also a sufficient one.
\end{remark}

\begin{proof}[Proof of Proposition \ref{deviats}]
Starting from stationarity allows us to control the distance between the first and last particle at all time. 
In particular we have 
\begin{equation}\label{staz}\begin{split}
\bbE\left[ \hat \eta_{n+1}(t)-\hat \eta_{1}(t)\right]&=\bbE\left[ \hat \eta_{n+1} (0)-\hat \eta_{1}(0)\right]= \sum_{i=1}^n \frac{1}{1-\mu_i}
=\frac{1}{1-\gb} \sum_{i=1}^n \frac{1}{1-\gl^{-i}}\\
&\le \frac{1}{1-\gb} \left(n+  \frac{C}{\gl-1}\log \left(\min(  n,  |\gl-1|^{-1} ) \right)\right),
\end{split}\end{equation}
for some universal constant $C$.
By union bound, the probability in the l.h.s.\ of \eqref{eq:woof} is smaller than
\begin{multline}
 \bbP\left[\hat \eta_{n+1}(t) \le t\gb b- A \sqrt{bt}\right]\\+ 
 \bbP\left[\hat \eta_{n+1}(t)
 -\hat \eta_1(t)\ge  \frac{A}{1-\gb}\left[n+b^{-1} \log \min(n,b^{-1})\right]\right].
\end{multline}
The first term is small because the expectation and the variance of $\hat \eta_{n+1}(t)$ are equal to $t\gb b$. 
The second can be shown to be going to zero with $A$ using \eqref{staz} and Markov's inequality for $\hat \eta_{n+1}(t)-\hat\eta_1(t)$.
\end{proof}

\subsection{Proof of Proposition \ref{lerimo} in the case $\alpha=0$}\label{Subsec:alpha0}

We restate and prove the result in this special case (observe that the result for $R_N$ is trivial for $\alpha=0$).

\begin{proposition}
Assume that $\alpha=0$ and \eqref{largebias} holds. We have for any $C>0$ and any $\gep>0$
\begin{equation}
\lim_{N\to \infty} \sup_{t\in [0,C b_N^{-1} N]} \bbP\left[ |L_N(t)- b_N t|\ge \gep N \right]=0.
\end{equation}
\end{proposition}

\begin{proof}
First let us remark that the convergence
\begin{equation}
\lim_{N\to \infty} \sup_{t\in [0,C b_N^{-1} N]} \bbP\left[ L_N(t)\ge b_N t + \gep N \right]=0
\end{equation}
follows from the fact that the first particle is stochastically
dominated by a simple random walk with bias $b_N \gg N^{-1}$ on the segment, starting from position $1$.

It remains to prove that 
\begin{equation}\label{srup}
\lim_{N\to \infty} \sup_{t\in [0,C b_N^{-1} N]} \bbP\left[ L_N(t)\le  b_N t-  \gep N \right]=0.
\end{equation}
We provide the details for the most important case $C=1$, and then we briefly explain how to deal with the case $C>1$.\\
We couple $\eta^{\wedge}(t)$ with the system $\hat \eta(t)$ of the previous section, choosing $n=k$ and $\gb=1-(\gep/2)$.
  The coupling is obtained by making  the $i$-th particle in both processes try to jump at the same time (for $i\in \lint 1, k \rint$) with rate $p$ and $q$, and rejection of the moves occurs as consequences of the exclusion rule or boundary condition (for $\eta^{\wedge}$ only).
  Initially of course we have 
  \begin{equation}\label{initial}
 \forall i \in \lint 1, n \rint, \quad  \eta^{\wedge}_{i}(0)\ge \hat \eta_i(0).
 \end{equation} 
 because of the choice of the initial condition for $\hat \eta$ (recall that by definition $\eta^{\wedge}_i(0)=i$).
 The boundary at zero, and the presence of one more particle on the right in $\hat \eta$ gives $\eta^\wedge$ only more pushes towards the right, so that the ordering is preserved at least until $\hat \eta_{n+1}$ reaches the right side of the segment and the effect of the other 
 boundary condition starts to be felt:
\begin{equation}
\eta^\wedge_{i}(t)\ge \hat \eta_i(t),\quad  \forall i \in \lint 1, n \rint, \forall t\le \cT
\end{equation}
where 
$\cT:= \inf \{t \ge 0 \ : \ \hat \eta _{n+1}(t)=N+1  \}$.

\medskip
Using the assumption \eqref{largebias}, a second moment estimate and the fact that $\gb<1$, we have 
$$\lim_{N\to \infty} \bbP[ \cT\le b^{-1}_N N] =0,$$ and hence
\begin{equation}
\lim_{N\to \infty} \sup_{t\le b^{-1}_N N} \bbP\left[\eta^\wedge_1(t)\le \hat \eta_1(t)\right]=0.
\end{equation}
Therefore, it suffices to control the probability of $\hat{\eta}_1(t) \le b_N t - \gep N$. Observe that the assumptions $(k_N/N)\to 0$ and \eqref{largebias} imply that for any given $A>0$, for all $N$ sufficiently large and for any 
$t\le b^{-1}_N N$ we have
$$\left( \sqrt{b_N t}+ \frac{1}{1-\gb}\left[k_N+b_N^{-1}\log\min(k_N , b^{-1}_N) \right] \right)\le \gep\frac{N}{2 A}.$$
Furthermore $(1-\beta)b_Nt \le \gep N/ 2$ for all $t\in [0,b_N^{-1} N]$. Thus applying Proposition \ref{deviats} we obtain that for $N$ sufficiently large 
\begin{multline}
\sup_{t\in [0,b_N^{-1} N]}   \bbP\left[ \hat \eta_1(t)\le  b_N t-\gep N\right]
\\
 \le  \sup_{t\in [0,b_N^{-1} N]}  \bbP\left[ \hat \eta_1(t)\le  \gb b_N t- A\left( \sqrt{b_N t}+ \frac{k_N+b_N^{-1}\log\min(k , b^{-1}_N)}{1-\gb}\right) \right]\le \delta.
\end{multline}
where $\delta$ can be made  arbitrarily small by choosing  $A$ large. This concludes the proof of \eqref{srup} for $C=1$. To treat the case $C>1$, it suffices to shift the particle system $\hat{\eta}(0)$ to the left by $\lfloor (C-1)b_N^{-1} N \rfloor$ and to apply the same arguments as before so we omit the details.
\end{proof}

 \subsection{Proof of Proposition \ref{lerimo} in the case $\alpha\in(0,1)$}\label{Subsec:alphaNon0}
 
 The roles of $L_N$ and $R_N$ being symmetric, we only need to prove the result for $L_N$ (but we do not assume here that $\alpha\le 1/2)$.
A direct consequence of Proposition \ref{prop:lidro} is that for all $s\in \bbR$  and $\gep>0$ we have 
\begin{equation}
\lim_{N\to \infty}\bbP\left[ \eta^{\wedge}_1(b_N^{-1} N s ) \ge N\left(\ell_{\alpha}(s)+\gep\right)\right] = 0.
\end{equation}
Hence to conclude  we want to prove that 
 \begin{equation}\label{toprove}
 \lim_{N\to \infty}\bbP\left[ \eta^{\wedge}_1(b_N^{-1} N s ) \le N\left(\ell_{\alpha}(s)-\gep\right) \right]=0.
 \end{equation}
 For the remainder of the proof $s$ and $\gep$ are considered as fixed parameters.
 We set $\delta \in (0,\alpha)$, and $n=\lceil \delta N \rceil$. To prove \eqref{toprove}, we are going to compare $(\eta^{\wedge}_i)_{i=1}^n$
 to the particle system considered in Section \ref{sec:auxi}.
 
 First we observe that as  a consequence of Proposition \ref{prop:lidro}, we have, for any  $T>0$ 
 \begin{equation}\label{camarch}
 \lim_{N \to \infty} \bbP \left[ \exists t\in [0, T], \   \eta^{\wedge}_{n+1}(b^{-1}_N Nt)\le N \ell_{\alpha}(t) \right]=0
 \end{equation}
We define the process $\hat \eta$ as in Section \ref{sec:auxi} with $\gb=1-\gep/(2s)$ but with a shifted initial condition. 
More precisely we set
$$\hat \eta_{n+1}(0)= N\left(\ell_{\alpha}(s)-s \right)\le 0,$$
and choose the initial particle spacings to be independent and with geometric distributions given by \eqref{geom}.
As \eqref{initial} is satisfied, we can couple the two processes in such a way that 
\begin{equation}
  \forall i \in \lint 1, n \rint\;, \forall t\le \cT', \quad  \eta^{\wedge}_{i}(t)\ge \hat \eta_i(t),
\end{equation}
where $\cT':= \inf\{ t \ :  \hat \eta_{n+1}(t)=\eta^{\wedge}_{n+1}(t)\}$. It is a simple exercise to show that for every $T>0$  the position of $\hat \eta_{n+1}$ satisfies the following law of large numbers
\begin{equation}
 \lim_{N\to \infty}\bbP\left[ \sup_{t\in [0,T]}  \left|\frac{\hat \eta_{n+1}(b_N^{-1}N t)}{N}-\left(\ell_{\alpha}(s)-s \right)-\gb t\right|\ge \kappa  \right]=0,\quad \forall \kappa >0\;,
\end{equation}
which, combined with \eqref{camarch}, yields
\begin{equation}
 \lim_{N\to \infty} \bbP[\cT' \ge b^{-1}_N N s]=1.
 \end{equation}
 and thus we only need to prove  \eqref{toprove} with $\eta^{\wedge}_1$ replaced by $\hat \eta_{1}$. More precisely 
 we prove that given $\kappa>0$, one can find $\delta$ sufficiently small such that 
 \begin{equation}\label{provex}
  \bbP\left[ \hat\eta_1(b_N^{-1} N s ) \le N\left(\ell_{\alpha}(s)-\gep\right) \right]\le \kappa.
 \end{equation}

 Using Proposition \ref{deviats} for $t=b^{-1}_N N s$  and $A= \delta^{-1/2}$ and taking into account the new initial condition, 
 the probability of the event 
 $$ \left\{ \hat \eta_1(b^{-1}_N N s)\le N\left( \ell_{\alpha}(s)-\frac{\gep}{2} \right)- \delta^{-1/2}\left( \sqrt{Ns}+ \frac{2s}{\gep}[\delta N + b^{-1}_N \log b^{-1}_N] \right) \right\}$$
 has a probability which can be made arbitrarily small if $\delta$ is chosen sufficiently small.
 We can then conclude that \eqref{provex} holds by observing that for $\delta$ sufficiently small and $N$ sufficiently large 
 $$\delta^{-1/2}\left( \sqrt{Ns}+ \frac{2s}{\gep}[\delta N + b^{-1}_N \log b^{-1}_N] \right)\le \gep N/2.$$

  \qed

\section{Lower bound on the mixing time for small biases}\label{Sec:LBSB}

Until the end of the section, we assume that the small bias assumption \eqref{smallbias} holds.

\medskip

Let us set
$s_{\delta}(N):= (1-\delta) \log k / (2\gap_N)$.
We show that at time $s_{\delta}$, equilibrium is not reached if one starts from one of the extremal conditions (some moderate efforts allow to replace $\max$ by 
$\min$ in the statement of the proposition).
\begin{proposition}\label{lbsb}
 When assumption \eqref{smallbias} holds, we have
 \begin{equation}
  \lim_{N\to \infty} \max_{\zeta\in \{\vee,\wedge\} } \| \bbP( h^{\zeta}_{s_{\gd(N)}}\in \cdot )-\pi_{N,k}\|_{TV}=1
 \end{equation}
As a consequence for every $\gep\in (0,1)$, $\Tm^{N,k}(\gep)\ge s_{\delta}(N)$ for $N$ sufficiently large. 
\end{proposition}

The method to obtain a lower bound on the mixing time for small biases is similar to the one used in the symmetric case (see \cite[Section 3.3]{Wil04}),
and is  based on the control of the two first moments of $f_{N,k}(h^{\wedge}_t)-f_{N,k}(\zeta)$ where $\zeta$ is independent of $h^\wedge_t$ and distributed according to $\pi_{N,k}$:
if at time $t$ the mean of $f_{N,k}(h^\wedge_t)-f_{N,k}(\zeta)$ is much larger than its standard deviation, then the system is not at equilibrium (cf  \cite[Proposition 7.12]{LevPerWil}).

We present estimates for the first two moments that we prove at the end of the section.
This first moment bound is elementary.
\begin{lemma}\label{lem:firstmom}
We have 
\begin{equation}
f_{N,k}(\wedge)-f_{N,k}(\vee)\ge  \frac{1}{8} \gl^{(k-N)/2} Nk\;,
\end{equation}
and as a consequence, for every $t\ge 0$
\begin{equation}\label{lunoulot}
\max( f_{N,k}(\wedge), -f_{N,k}(\vee)) \ge \frac{1}{16} \gl^{(k-N)/2} Nk\;.
\end{equation}
\end{lemma}

The second moment estimates rely on the control of a martingale bracket.
\begin{lemma}\label{Lemma:BoundVar}
For all $t\ge 0$, $N\ge 1$ and all $k\in \lint 1, N/2\rint$
we have 
\begin{equation}
\var(f_{N,k}(h_t^\wedge)\le \frac{k \lambda^k}{2 \gap_N},
\end{equation}
The same bound holds for $\var(f_{N,k}(h_t^\vee))$ and $\var_{\pi_{N,k}}(f_{N,k})$.
\end{lemma}

\begin{proof}[Proof of Proposition \ref{lbsb}]
Let us assume for simplicity (recall \eqref{lunoulot}) that 
\begin{equation}
f_{N,k}(\wedge)\ge \frac{1}{16} \gl^{(k-N)/2} Nk.
\end{equation}
(if not we apply the same proof to $f_{N,k}(\vee)\le -\frac{1}{16} \gl^{(k-N)/2} Nk$). By the material in Section \ref{sec:eigen}, we have 
\begin{equation}\label{cleup}
\bbE\left[ f_{N,k}(h^\wedge_t)\right] \ge  \frac{1}{16} e^{-\gap_N t} \gl^{(k-N)/2} Nk.
\end{equation}
Applying \cite[Proposition 7.12]{LevPerWil} for the probability measures $P_t^{N,k}(\wedge, \cdot)$ and $\pi_{N,k}$ and the function $f_{N,k}$ (recall that $E_{\pi_{N,k}}[f_{N,k}]=0$), we obtain that 
\begin{equation}
 \|P_t^{N,k}(\wedge, \cdot)-\pi_{N,k}  \|_{TV} \ge 1-\frac{2 \left(\var(f_{N,k}(h_t^\wedge))+\var_{\pi_{N,k}}(f_{N,k}) \right)}{\bbE\left[f_{N,k}(h_t^\wedge)\right]^2}\;.
\end{equation}
Using Lemma \ref{Lemma:BoundVar} and \eqref{cleup}, we obtain that 
 \begin{equation}
  \frac{\var(f_{N,k}(h_t^\wedge))+\var_{\pi_{N,k}}(f_{N,k})}{\bbE\left[f_{N,k}(h_t^\wedge)\right]^2}\le \frac{ 16^2 e^{2\gap_N t} \gl^N}{\gap_N N^2 k}.
 \end{equation}
Now if we apply this inequality at time  $s_{\delta} = (1-\delta) \log k / (2\gap_N)$,
then we obtain for any given $\gep > 0$ and all $N$ sufficiently large
\begin{equation}
d(t_1)\ge 1-2\frac{16^2 \gl^N }{k^{\delta}\gap_N N^2}\ge 1-\gep.
\end{equation}
where we used the small bias assumption \eqref{smallbias}. This yields $\Tm^{N,k}(\gep)\ge s_{\delta}$.
\end{proof}
\begin{proof}[Proof of Lemma \ref{lem:firstmom}]
We have
\begin{align*}
f_{N,k}(\wedge) - f_{N,k}(\vee) &= \sum_{x=1}^{N-1} \sin \left(\frac{x\pi}{N}\right) \frac{\lambda^{\frac12 \wedge(x)} - \lambda^{\frac12 \vee(x)}}{\lambda-1}\\
&\ge  \sum_{x=1}^{N-1} \sin \left(\frac{x\pi}{N}\right) \lambda^{\frac12 \vee(x)} \frac{\wedge(x)-\vee(x)}{2},
\end{align*}
where the last inequality is obtained similarly to \eqref{linex}.
Since $\vee(x)\ge k-N$ for all $x$ and $\wedge(x)-\vee(x) \ge k$ for all $x\in \{N/4,\ldots,3N/4\}$, we conclude that 
$$ \sum_{x=1}^{N-1} \sin \left(\frac{x\pi}{N}\right) \lambda^{\frac12 \vee(x)} \frac{\wedge(x)-\vee(x)}{2} \ge \frac{\sqrt 2}{2}\gl^{\frac{k-N}{2}}\frac{N k}{4}.$$
\end{proof}

\begin{proof}[Proof of Lemma \ref{Lemma:BoundVar}]
By the material in Section \ref{sec:eigen}, we know
$$ M_t := f_{N,k}(h^\wedge_t) e^{\gap_N t}$$
is a martingale. Its predictable bracket is given by
\begin{align*}
\langle M_\cdot \rangle_t &= \int_0^t \sum_{x=1}^{N-1} \lambda^{h_s^\wedge(x)} \sin\Big(\frac{\pi x}{N}\Big)^2 e^{2\gap_N s}\\
&\quad\times\Big(p_N\ind_{\{\Delta h_s^\wedge(x) < 0\}} \gl^{-2} + q_N \ind_{\{\Delta h_s^\wedge(x) > 0\}}\Big) ds\;,
\end{align*}
and $M_t^2 - \langle M_\cdot \rangle_t$ is again a martingale. This yields the identity
$$\var(f_{N,k}(h_t^\wedge)) = e^{-2\gap_N t}\E\big[\langle M_\cdot \rangle_t\big].$$
To bound the predictable bracket of $M$, let us observe that 
the number of possible particle transitions to the right and to the left (the number of sites $x$ such that $\Delta h_s^\wedge(x)<0$, resp. $>0$) is bounded by $k$, and that for any $x$ and $\zeta\in \gO_{N,k}$ we have $\lambda^{\zeta(x)} \le \lambda^{k}$. 
Therefore, we obtain the bound
\begin{align*}
\E\big[\langle M_\cdot \rangle_t\big] &\le \int_0^t e^{2\gap_N s} ds\, \lambda^k \sum_x \ind_{\{\Delta h_s^\wedge(x) \ne 0\}}\le k\lambda^k \frac{e^{2\gap_N t}}{2\gap_N}\;,
\end{align*}
which yields the asserted bound. The case of $h^{\vee}_t$ is treated in the same manner by symmetry.
Since the distribution of $h^\wedge_t$ converges to $\pi_{N,k}$ when $t$ tends to infinity we deduce that
$$ \var_{\pi_{N,k}}(f_{N,k}) = \lim_{t\rightarrow\infty} \var(f_{N,k}(h_t^\wedge))\;,$$
which allows to conclude.
\end{proof}

\section{Upper bound on the mixing time for small biases}\label{Sec:UBSB}

Until the end of the section we assume that the small bias assumption \eqref{smallbias} holds and that the different initial conditions are coupled using the monotone grand coupling $\bbP$ defined in Appendix \ref{Appendix:Coupling}. We set for all $\delta > 0$
$$t_{\delta}(N):=(1+\delta)\frac{\log k}{2\gap_N}.$$
Recall the definition of the merging time $\tau$ from \eqref{eq:mergin}.

\begin{proposition}\label{oldform}
 Assume that \eqref{smallbias} holds. We have 
 \begin{equation}\label{laforme}
\lim_{N\to \infty}  \bbP[\tau < t_{\delta}(N)]=1.
 \end{equation}
As a consequence, for every $\gep>0$ and all $N$ sufficiently large, $\Tm^{N,k}(\gep)\le t_{\delta}(N)$.
 
\end{proposition}

Recall (see the paragraph after \eqref{eq:grand}) that $h^{\pi}_t$ denotes the chain with stationary initial condition.
For practical reasons, it is simpler to couple two processes when at least one of them is at equilibrium.
We thus prove \eqref{laforme} by showing that $\lim_{N\to \infty}\bbP[\tau_i < t_{\delta}(N)]=1$ for $i\in \{1,2\}$ where 
\begin{equation}
 \tau_1:=\inf\{ t>0 \ : \  h^{\wedge}_t=h^{\pi}_t \}\quad \text{ and }\quad \tau_2:=\inf\{ t>0 \ : \  h^{\vee}_t=h^{\pi}_t \}. 
\end{equation}
The argument being completely symmetric, we focus only on $\tau_1$.
As in Sections \ref{preuv1} and \ref{preuv2}, we interpret $\tau_1$ as the time
at which the weighted area $A_t$ between the maximal and equilibrium interface vanishes
\begin{equation}\label{defAt}
A_{t}:= \frac{f^{(0)}_{N,k}(h^{\wedge}_t)-f^{(0)}_{N,k}(h^{\pi}_t)}{\delta_{\min}(f^{(0)}_{N,k})}
=\gl^{\frac{N-k}{2}}\sum_{x=1}^{N-1} \frac{\lambda^{\frac12 h_t^\wedge(x)}-\lambda^{\frac12 h_t^\pi(x)}}{\lambda - 1} \;.
\end{equation}
A simple computation based on the identity \eqref{Eq:V} shows that $A$ is a supermartingale.

\medskip

While in the large bias case (Section \ref{Sec:UBLB}) the choice of the grand coupling does not matter, here it is crucial to use a coupling which maximizes in a certain sense the fluctuation of the weighted area $A_t$, so that this process reaches zero as quickly as possible.
The coupling defined in Appendix \ref{Appendix:Coupling} makes the transitions for the two processes $h^{\wedge}$ and $h^{\pi}$ as independent as possible (some transitions must occur simultaneously for the two processes in order to preserve monotonicity). 

\medskip

We consider $\eta >0$ small and introduce the successive stopping times $\cT_i$ by setting
$$\cT_0 := \inf\big\{t\ge t_{\delta/2}: A_t \le k^{\frac12 - \frac \gd 5} N\big\}\;,$$
and
$$\cT_{i} := \inf\big\{t\ge \cT_{i-1}: A_t \le k^{\frac12 - i\eta - \frac \gd 5} N\big\}\;,\quad i\ge 1\;.$$
We also set for coherence $\cT_{\infty}:=\max(\tau_1,t_{\delta/2})$ the first time at which $A_t$ reaches $0$. Notice that some of these stopping times may be equal to $t_{\delta/2}$.\\
Set $T_N:=\min( b^{-2}_N, N^2)$. To prove Proposition \ref{oldform}, we show first  that $A_t$ shrinks to 
$k^{\frac12-\frac \gd 5} N$ by time $t_{\delta/2}$ and then that it only needs an extra time $2T_N$ to reach $0$. The second step is performed by controlling each increment 
$\gD \cT_i:=\cT_i-\cT_{i-1}$ separately for each $i$ smaller than some threshold $K:=\lceil 1/(2\eta)\rceil$.

\begin{lemma}\label{newform}
Given $\delta$, if $\eta$ is chosen small enough and $K:=\lceil 1/(2\eta)\rceil$, we have 
$$ \lim_{N\to \infty}\P\left( \{ \cT_{0} = t_{\delta/2}\}\cap \left(\bigcap_{i=1}^K 
\{\gD\cT_{i}\le 2^{-i} T_N\}\right)\cap\{\cT_{\infty}-\cT_{K}\le T_N \}\right)=1 \;.$$
\end{lemma}
\noindent Note that on the event defined in the lemma and for all $N$ large enough, we have
$$\tau_1\le \cT_{\infty}\le t_{\delta/2}+2T_N\le t_{\delta}.$$
Hence Proposition \ref{oldform} follows as a direct consequence. 

\medskip

The bound on $\cT_0$ is  proved in Section \ref{s:contract}, while that of on $\cT_{\infty}-\cT_{K}$ follows from Lemma \ref{Lemma:AN} in Section \ref{s:dif},
the case of the other increments is more delicate and is detailed in Section \ref{s:inter}.

\subsection{Contraction estimates}\label{s:contract}
The approach used in the first step bears some similarity with the one used in Section \ref{preuv1}, the notable difference being that 
\eqref{eq:contract0} is not sufficient here and we must work a bit more to show that $\bbE[A_t]$ decays with rate $\gap_N$.

\begin{lemma}\label{Lemma:T0}
Given $\delta>0$ we have $\P\big( \cT_0 > t_{\delta/2}\big) \to 0$ as $N\to\infty$.
\end{lemma}

\begin{proof}
 Note that $a(t,x):= \bbE\left[\frac{\lambda^{\frac12 h_t^\wedge(x)}-\lambda^{\frac12 h_t^\pi(x)}}{\lambda - 1} \right]$
 is a solution of the equation
 $$\partial_t a = (\sqrt{pq}\, \gD-\varrho)a\;,$$
 with $a(t,0)=a(t,N)=0$. Diagonalising the operator on the right hand side, see Subsection \ref{sec:eigen}, 
 we get the following bound on the $\ell^2$-norm of the solution:
 $$\sum_{x=1}^{N-1}a(t,x)^2 \le e^{-2\gap_N t}\sum_{x=1}^{N-1}a(0,x)^2,$$
 and using Cauchy-Schwartz inequality
 we obtain 
 \begin{equation}
 \gl^{\frac{k-N}{2}}\bbE[A_t] = \sum_{x=1}^{N-1}a(t,x)\le \sqrt{N} e^{-\gap_N t}\sqrt{\sum_{x=1}^{N-1}a(0,x)^2}\le  2e^{-\gap_N t} Nk \gl^{k/2}\;.
 \end{equation}
 Since $\lambda^{N/2}$ is, by the small bias assumption, asymptotically smaller than any power of $k$, Markov's inequality concludes the proof.
\end{proof}

\subsection{Diffusion estimate after time $t_{\delta/2}$}\label{s:dif}

Now this part is much more delicate than Section \ref{preuv2}. The reason being that since $T_N$ is extremely close to $t_{\delta}$,
we need very accurate control on the derivative of the predictable bracket of $A_t$.  
Our first task is to use Proposition \ref{prop:solskjaer} in order to control the increment of the bracket of $A$ in between the $\cT_i$'s.
Let us set 
\begin{equation}
\Delta_i \langle A\rangle := \langle A_\cdot\rangle_{\cT_{i}} - \langle A_\cdot\rangle_{\cT_{i-1}},\qquad\Delta_{\infty} \langle A\rangle := \langle A_\cdot\rangle_{\cT_{\infty}} - \langle A_\cdot\rangle_{\cT_{K}}\;,
\end{equation}

and consider the event
$$\cA_N:= \left\{ \forall i \in \lint 1,K \rint, \quad    \Delta_i \langle A\rangle \le k^{1 - 2(i-1)\eta - \frac \gd 4} N^2  \right\}\cap 
\left\{ \Delta_{\infty} \langle A\rangle \le T_N  \right\}$$

\begin{lemma}\label{Lemma:AN}
 We have $\lim_{N\to \infty} \P[\cA_N^{\cc}]=0$.
\end{lemma}

\begin{proof}
  We apply Proposition \ref{prop:solskjaer}-(ii) to $(A_{t+\cT_{i-1}})_{t\ge 0}$, with $a=k^{\frac12 - (i-1)\eta - \frac \gd 5}N$, $b=k^{\frac12 - i\eta - \frac \gd 5}N$. We obtain that for all $N$ sufficiently large and every $i\le K$
 $$\bbP[  \Delta_i \langle A\rangle\ge k^{1 - 2(i-1)\eta - \frac \gd 4} N^2]\le k^{-\delta/100}.$$ 
Applying the same proposition to $(A_{t+\cT_{K}})_{t\ge 0}$ with $a=k^{- \frac \gd 5}N$ and $b=0$, we obtain 
\begin{equation}
\bbP[  \Delta_\infty \langle A\rangle\ge T_N ]\le 8 N k^{-\frac {\gd}{5}}(T_N)^{-1/2},
\end{equation}
and the r.h.s.\ tends to zero by assumption \eqref{smallbias}.
\end{proof}

The next step is to compare $\Delta_i \langle A\rangle$ with $\cT_i-\cT_{i-1}$.
For the last increment this is easy: We have  $\partial_t \langle A_\cdot\rangle\ge 1$ for any $t\le \cT_{\infty}$ (from our construction 
$A$ changes its value at rate at least $1$, and its minimal increment in absolute value is $1$). We have thus $\cT_{\infty}-\cT_K \le \gD_{\infty} \langle A \rangle$, and 
thus when $\cA_N$ holds we have
\begin{equation}\label{Eq:Tinf}
\cT_{\infty}-\cT_K \le T_N\;.
\end{equation}

\subsection{Control of intermediate increments}\label{s:inter}

For all other increments we have to use a subtler control of the bracket.
Let us set 
 $$\bH(t):= \gl^{\frac{N-k}{2}}\max_{x\in \lint 0,N \rint} \frac{\gl^{\frac12 h^{\wedge}_t(x)}- \gl^{\frac12 h^{\vee}_t(x)}}{\gl-1}\;,$$
which corresponds roughly (up to a multiplicative factor $\gl^{N}$) to the maximal height difference  $\max h^{\wedge}_t(x)-h^{\vee}_t(x)$ and 
thus provides a bound for  $\max_x h^{\wedge}_t(x)-h^{\pi}_t(x)$.\\
Recall $Q(\cdot)$ from Subsection \ref{Subsec:eqSmallBias}, and set $Q(h_t^\pi) := Q(\eta_t^\pi)$ where $\eta_t^\pi$ is the particle configuration associated with $h_t^\pi$.
 
\begin{lemma}\label{Lemma:BrackDeriv}
We have $\partial_t \langle A_\cdot\rangle \ge \frac{   A_t }{6 \bH(t) Q(h_t^{\pi})}$.
 \end{lemma}

 \begin{proof}
 As mentioned above, all the jumps of $A_t$ have amplitude larger  than or equal to $1$.
 Moreover, $A_t$ performs a jump whenever $h^{\pi}_t$ performs a transition while $h^{\wedge}_t$ does not, or when the opposite occurs.
  As any such transition occurs at rate larger than $q_N\ge 1/3$, only considering the transitions for $h^{\pi}_t$ , we obtain the following lower bound for the drift of 
  $\langle A_\cdot\rangle$
  (recall \eqref{laplace})
  \begin{equation}\label{loopz}
   \partial_t \langle A_\cdot\rangle \ge \frac{1}{3} \#\{ x\in \cC_t \ : \ \gD(h^{\pi}_t)(x)\ne 0\}=:\frac{1}{3}\#\cD_t
  \end{equation}
where
$$ \cC_t:= \big\{x\in\lint 1,N-1\rint \ : \ \exists y\in \lint x-1,x+1\rint,\   h^\wedge_t(y) > h^{\pi}_t(y) \big\}\;.$$
Now let $\lint a, b\rint$ be a maximal connected component of $\cC_t$, we claim that 
\begin{equation}\label{zaap}
 \#( \cD_t \cap \lint a,b \rint)\ge \max\left( \left\lfloor \frac{b-a}{Q(h^{\pi}_t)}\right\rfloor, 1 \right)
 \ge \frac{b-a}{2Q(h^{\pi}_t)}.
\end{equation}
To check this inequality, notice that $\#( \cD_t \cap \lint a,b \rint)\ge 1$ because $h^{\pi}_t$ cannot be linear on the whole segment $\lint a, b\rint$. On the other hand, considering the particle configuration associated to $h^\pi_t$ and decomposing the segment $\lint a,b\rint$ into maximal connected components containing either only particles or only holes, we see that any two consecutive components corresponds to a point in $\cD_t$: since $Q(h^{\pi}_t)$ is an upper bound for the size of these components, we deduce that $\#( \cD_t \cap \lint a,b \rint)\ge \left\lfloor \frac{b-a}{Q(h^{\pi}_t)}\right\rfloor$.

\medskip

Now we observe that 
\begin{equation}\label{ziip}
 \gl^{\frac{N-k}{2}}\sum_{x=a}^b \frac{\gl^{\frac{h^{\wedge}_t(x)}{2}}-\gl^{\frac{h^{\pi}_t(x)}{2}}}{\gl-1} \le 
 (b-a) \bH(t). 
\end{equation}
Combining \eqref{zaap} and \eqref{ziip} and summing over all such intervals $\lint a,b\rint$, we obtain 
\begin{equation}
 A_t\le  2\,\#\cD_t \, \bH(t)\,Q(h^{\pi}_t),
\end{equation}
and \eqref{loopz} allows us to conclude.
\end{proof}
The last ingredient needed is then a bound on $\bH$: The proof of this lemma is postponed to Subsection \ref{Subsec:Max}. Recall that $t_0=\log k/(2\gap_N)$.

\begin{proposition}\label{Prop:BoundMax0}
For any $c >0$ we have
\begin{equation}
\lim_{N\to \infty} \sup_{t\ge t_0}  \P\Big(\bH(t)  >  k^{\frac12 + c}\Big) = 0\;.
\end{equation}
\end{proposition}

\begin{proof}[Proof of Lemma \ref{newform}] 
By Lemma \ref{Lemma:T0}, Lemma \ref{Lemma:AN} and Equation \eqref{Eq:Tinf}, we already know that
$$ \lim_{N\to \infty}\P\left( \{ \cT_{0} \le t_{\delta/2}\}\cap \{\cT_{\infty}-\cT_{K}\le T_N \}\right)=1 \;.$$
We define $\cH_N$ to be the event on which particles are reasonably spread and $\bH(t)$ is reasonably small for most of the times 
within the interval $[t_{\delta/2}, t_{\delta/2}+T_N]$,
\begin{equation}
 \cH_N:= \Big\{ \int^{t_{\delta/2}+T_N}_{t_{\delta/2}}  \ind_{\{\text{ $\bH(t) \le  k^{\frac12 + \frac{\delta}{80}}\} \cap \{Q(h_t^\pi) \le N k^{\frac{\delta}{80}-1}$ } \}} dt  \ge T_N (1-2^{-(K+1)}) \Big\}\;.
\end{equation}
By Markov's inequality, Proposition \ref{lem:dens} and Proposition \ref{Prop:BoundMax0}, we have 
$$ \lim_{N\to \infty}\P(\cH_N) = 1\;.$$
We now work on the event $\cH_N \cap \cA_N \cap \{\cT_0 \le t_{\delta/2}\}$ whose probability tends to $1$ according to Lemmas \ref{Lemma:T0}, \ref{Lemma:AN}.
We prove by induction that $\Delta\cT_j \le 2^{-j} T_N$ for all $j\in \lint 1,K\rint$. 
Let us reason by contradiction and let $i$ be the smallest integer such that $\Delta\cT_i > 2^{-i} T_N$.
We have
\begin{equation}\label{indd}
\Delta_i\langle A\rangle \ge \int_{\cT_{i-1}}^{\cT_{i-1}+2^{-i}T_N} \partial_t \langle A_\cdot
\rangle\ind_{\{\text{ $\cH(t) \le k^{\frac12 + \frac{\delta}{80}}\}\cap\{Q(h_t^\pi) \le N k^{\frac{\delta}{80}-1}$ } \}} dt\end{equation}
Now, Lemma \ref{Lemma:BrackDeriv} and the restriction with the indicator function provides a uniform lower bound on $\partial_t \langle A \cdot \rangle$.
 The assumption $\Delta\cT_j \le 2^{-j} T_N$ for $j<i$ implies that $\cT_{i-1}\le t_{\delta/2} + T_N(1-2^{-(i-1)})$, and thus the assumption that $\cH_N$ holds implies that the indicator in
 \eqref{indd} is equal to one 
 on a set of measure at least $2^{-i} - 2^{-(K+1)}\ge  2^{-(K+1)}$. All of this  implies that
 \begin{equation}
\Delta_i\langle A\rangle \ge \frac16 T_N 2^{-(K+1)} k^{1 -i\eta -\frac{\delta}{40} - \frac{\delta}{5}}
\end{equation}
On the other hand, since we work on $\cA_N$ we have $\Delta_i\langle A\rangle \le k^{1-2(i-1)\eta-\frac{\delta}{4}}N^2$ so that we get a contradiction as soon as $\eta$ is small enough compared to $\delta$.
\end{proof}

\subsection{Bounding the maximum}\label{Subsec:Max}

Recall the function $a_{N,k}$ defined in Subsection \ref{sec:eigen}. Set
$$ H_1(t,x) := \gl^{\frac{N-k}{2}}\frac{\lambda^{\frac12 h_t^\wedge(x)}-a_{N,k}(x)}{\lambda- 1}\;,\quad H_2(t,x) :=\gl^{\frac{N-k}{2}} \frac{\lambda^{\frac12 h_t^{\vee}(x)}-a_{N,k}(x)}{\lambda- 1}\;,$$
so that
$$ H_1(t,x)-H_2(t,x) = \gl^{\frac{N-k}{2}}\frac{\lambda^{\frac12 h_t^\wedge(x)}-\lambda^{\frac12 h_t^{\vee}(x)}}{\lambda- 1}\;.$$
For every $i=1,2$, we define
$$\bH_i(t):= \max_{x\in \lint 0,N\rint} |H_i(t,x)|\;.$$
Notice that $\bH(t) \le \bH_1(t) + \bH_2(t)$ so that Proposition \ref{Prop:BoundMax0} is a consequence of the following result.
\begin{proposition}\label{Prop:BoundMax}
For any $c >0$, there exists $c' >0$ such that for all $N$ large enough
$$ \sup_{t\ge t_0} \max_{i\in \{1,2\}} \P\Big(\bH_i(t)  >  k^{\frac12 + c}\Big) \le e^{- k^{c'}}\;.$$
\end{proposition}

The proof of this bound is split into two lemmas. First, we show that $H_i(t,\cdot)$ can not decrease too much.

\begin{lemma}\label{smsl}
We have for all $N$ sufficiently large, all $x\in \lint 1,N-1\rint$, all $t\ge 0$, every $i\in\{1,2\}$ and every $y\ge x$
\begin{equation}\label{zladiff}
H_i(t,y)-H_i(t,x)\ge -\frac{k^2(y-x)}{4N}.
\end{equation}
\end{lemma}
\begin{proof}
It is of course sufficient to prove that
$$H_i(t,x)-H_i(t,x-1)\ge -\frac{k^2}{4N}\;.$$
We have for any $\eta\in \gO^0_{N,k}$, setting $h=h(\eta)$,
 \begin{equation}
 \frac{\lambda^{\frac12 h(x)}-\lambda^{\frac12 h(x-1)}}{\lambda- 1}
 =\lambda^{\frac12 (h(x-1)-1)}\left(\eta(x)-\frac{1}{\sqrt{\gl}+1}\right).
 \end{equation}
 Note that $a_{N,k}(x)=E_{\pi_{N,k}}[ \gl^{h(x)/2}]$ where $a_{N,k}$ was defined in Section \ref{sec:eigen}. Hence using the fact that $\eta^{\wedge}_t(x)\ge 0$, we get (the same holds for $i=2$ and $h^\vee$):
 \begin{multline}\label{Eq:H1H1}
  H_1(t,x)-H_1(t,x-1)\\
  \ge 
  \gl^{\frac{N-k}{2}}\frac{ E_{\pi_{N,k}}\left[\lambda^{\frac{h(x-1)-1}{2}}\right]-\lambda^{\frac{h^{\wedge}_t(x-1)-1}{2}}}{\sqrt{\gl}+1}  
  -\gl^{\frac{N-k}{2}} E_{\pi_{N,k}}\left[ \lambda^{\frac{h(x-1)-1}{2}}\eta(x)\right].
 \end{multline}
By Proposition \ref{lem:dens} and the small bias assumption \eqref{smallbias}, we have for all $N$ large enough
$$\gl^{\frac{N-k}{2}} E_{\pi_{N,k}}\left[ \lambda^{\frac{h(x-1)-1}{2}}\eta(x)\right]\le \gl^{\frac{N}{2}}E_{\pi_{N,k}}[\eta(x)]\le \gl^{2N} \frac{k}{N}\le \frac{k^2}{8N}.$$
Regarding the first term on the r.h.s.~of \eqref{Eq:H1H1}, we simply notice that for $\zeta, \zeta'\in \gO_{N,k}$, we have $\zeta(x)-\zeta'(x)\le 2k$ so that
\begin{equation}
 \gl^{\frac{N-k}{2}}|\gl^{\frac{\zeta(x)}{2}}-\gl^{\frac{\zeta'(x)}{2}}|\le (\gl-1)\gl^{\frac{N}{2}} k \le \frac{k^{2}}{8N}\;.
\end{equation}
This is sufficient to conclude.
\end{proof}

Let us introduce the average of $H_i(t,\cdot)$ over a box of size $\ell =\ell_{N,k}= \lceil \frac{N}{k^2}\rceil$
$$ \bar{H}_i(t,y) := \frac1{\ell}\sum_{x=(\ell- 1)y+1}^{y\ell} H_i(t,x)\;.$$
 As a consequence of Lemma \ref{smsl} we have
\begin{equation}
\bH_i(t)= \max_{x\in \lint 0,N\rint} |H_i(t,x)|\le  \max_{y\in \lint 1,  N/\ell \rint} \big| \bar{H}_i(t,y) \big|+1 \;.
\end{equation}
The result is of course obvious when $\ell=1$.
For $\ell \ge 2$, let us briefly explain why $\max H_i(t,x)\le  \max |\bar{H}_i(t,y)|+1$ (the case for $-\min$ follows by symmetry).
If $x_{\max}$ is the smallest $x$ at which the $\max$ is attained, we must distinguish two cases 
\begin{itemize}
 \item [(A)]$x_{\max}>  \ell\left( \lfloor N/\ell \rfloor-1\right)+1\ge N-2\ell$,
in which case \eqref{zladiff} applied for $x_{\max}$  and $N$ implies that $H_i(t,x_{\max})\le 1$, 
\item[(B)] $x_{\max}\le \ell\left( \lfloor N/\ell \rfloor-1\right)+1$ in which case one can compare 
$H_i(t,x_{\max})$ with $\bar{H}_i(t,y)$ for the smallest $y$ such that $x_{\max}\le y(\ell- 1)+1$ using \eqref{zladiff} again.
\end{itemize}
Then, Proposition \ref{Prop:BoundMax} is a direct consequence of the following bound on the averages of $H_i$.
\begin{lemma}
For any $a >0$, there exists $a' >0$ such that for all $N$ large enough
$$\sup_{t\ge t_0} \max_{i\in\{1,2\}} \P\Big( \max_{y\in \lint 1,  N/\ell \rint} |\bar{H}_i(t,y)|  >  k^{\frac12 + a}\Big) \le e^{- k^{a'}}\;.$$
\end{lemma}
\begin{proof}
We treat in details the bound of $\bar{H}_1$, since the bound of $\bar{H}_2$ follows from the same arguments. 
Using a decomposition of $\gl^{\frac{k-N}{2}} H_1(t,\cdot)$, which is a solution of \eqref{Eq:V}, on the basis of eigenfunction of the Laplacian formed by $\sin(i \pi \cdot)$, $i=1,\ldots,N-1$, we obtain the following expression for the mean
\begin{equation*}
\bbE[H_1(t,x)]= \gl^{\frac{N-k}{2}} \sum_{i=1}^{N-1} \frac2{N} e^{-\gamma_i t} f^{(i)}_{N,k}(\wedge) \sin\big( \frac{i \pi x}{N}\big),
\end{equation*}
and the fluctuation around it
\begin{multline}\label{2dterm}
  H_1(t,x)- \bbE[H_1(t,x)]\\
 =\gl^{\frac{N-k}{2}} \sum_{i=1}^{N-1} \frac{2}  {N} \big(f^{(i)}_{N,k}(h^\wedge_t)-e^{-\gamma_i t} f^{(i)}_{N,k}(\wedge)\big) \sin\big(\frac{i \pi  x}{N}\big).
\end{multline}
We bound separately the contributions to $\bar{H}_1$ coming from these  two terms.
We start with the mean. Since $\gl^{\frac{1}{2}\wedge(y)} \ge a_{N,k}(y) \ge\gl^{\frac{1}{2}  \vee(y)}$ for every $y\in\lint 0,N\rint$, we have (recall \eqref{linex})
\begin{align*}
|f^{(i)}_{N,k}(\wedge)| \le \sum_{y=1}^{N-1} \frac{\gl^{\frac12 \wedge(y)} -  a_{N,k}(y)}{\gl -1} 
\le \sum_{y=1}^{N-1} \lambda^{\frac12 \wedge(y)}\,\frac{\wedge(y)-\vee(y)}{2}
\le \lambda^{\frac{k}{2}} k N\;.
\end{align*}
Since $\lambda^{k/2}$ is negligible compared to any power of $k$, we deduce that for all $a>0$ and all $t\ge t_{0}$ we have for all $N$ large enough
$$ \sup_{x\in \lint 0,N\rint} \left|\sum_{i=1}^{N-1} \frac2{N} e^{-\gamma_i t} f^{(i)}_{N,k}(\wedge) \sin\big(i \pi \frac{x}{N}\big)\right| \le k^{(1+a)/2} 
\sum_{i=1}^{N-1} e^{(\gamma_1-\gamma_i) t_0}\;.$$
Notice that there exists $c>0$ such that for all $i\ge 2$ and all $N$ large enough
\begin{equation}\label{Eq:BoundGamma} \gamma_i - \gamma_1 \ge  c \frac{i^2}{N^2}\;.
\end{equation}
In addition, we have $N^2 \gap_N \ll (\log k)^2$ by the small bias assumption \eqref{smallbias}, so that we get for $i\ge 2$
$$ e^{(\gamma_1-\gamma_i) t_{0}} \le e^{- c \frac{i^2}{N^2} \frac{\log k}{2 \gap_N}} \le e^{-c' \frac{i^2}{\log k}} \le e^{-c' \frac{i}{\log k}}\;,$$
so that for all $N$ large enough we have
$$ \sum_{i=2}^{N-1} e^{(\gamma_1-\gamma_i) t_{0}} \le \sum_{i=2}^{N-1} e^{-c' \frac{i}{\log k}} \le C \log k \;.$$
Recall that $\ell = \lceil N/k^2 \rceil $. Putting everything together and using assumption \eqref{smallbias}, we get that given $a>0$ for $N$ sufficiently large and all values of $y$ we have
\begin{equation}\label{stimean}
\bbE[\bar H_1(t,y)]  =
\frac{ 2\gl^{\frac{N-k}{2}} }{N\ell} \Big|\sum_{x=\ell(y-1)+1}^{\ell y}\sum_{i=1}^{N-1}  e^{-\gamma_i t} f^{(i)}_{N,k}(\wedge) 
\sin\big(i \pi \frac{x}{N}\big)\Big| 
\le  \frac{1}{2} k^{\frac12+a} \;.
\end{equation}

\medskip

We turn to the contribution coming from the second term \eqref{2dterm}. To that end, we rewrite it 
 in the form
\begin{equation}
 \bar H_1(s,y)-\bbE\left[  \bar H_1(s,y) \right] = \gl^{\frac{N-k}{2}}\sum_{i=1}^{N-1} \big(f^{(i)}_{N,k}(h_t^\wedge) - e^{-\gamma_i s}f^{(i)}_{N,k}(\wedge)\big) \Phi_{y,i}\;.
 \end{equation}
 where (the second expression being obtained by summation by part)
\begin{multline}
\Phi_{y,i} = \frac2{N \ell } \sum_{x=y(\ell-1)+1}^{y\ell} \sin\big( \frac{i\pi x}{N}\big)\\
=  \frac1 {N \ell \sin \left(\frac{i\pi }{2N}\right)} \left[\cos\left( \frac{[2y(\ell-1)+1]i\pi }{2N}\right)- 
\cos\left( \frac{[2y\ell+1]i\pi }{2N}\right)    \right]\;.
\end{multline}
Note that for all $N\ge1$, $y$ and $i$ we have
$$ \big|\Phi_{y,i}\big| \le 2\min\left( \frac 1 N, \frac{1}{ i \ell }\right) \;.$$
Now let us fix $t$ and $y$ and introduce the martingale
$$ N_s^{(t,y)} = \gl^{\frac{N-k}{2}} \sum_{i=1}^{N-1} e^{\gamma_i (s-t)} \big(f^{(i)}_{N,k}(h_s^\wedge) - e^{-\gamma_i s}f^{(i)}_{N,k}(\wedge)\big) \Phi_{y,i}\;,\quad s\in [0,t]\;.$$
which satisfies 
$$N_0^{(t,y)}=0 \text{ and } N^{(t,y)}_t=\bar H_1(t,y)-\bbE\left[  \bar H_1(t,y) \right].$$
We wish to apply Lemma \ref{lem:expobd} to the martingale $N^{(t,y)}$: the maximal jump rate of this process is bounded by $k$ 
and the maximal amplitude of the jump (cf.\ the notations introduced in Appendix \ref{lapC}) 
satisfies 
$$\forall s \in [0,t], \quad S(s)\le \gl^N \sum_{i=1}^{N-1} e^{\gamma_i(s-t)} \big|\Phi_{y,i}\big|\;.$$
Using that as a consequence of \eqref{Eq:BoundGamma} we have $\gamma_i\ge c i^2 N^{-2}$ for all $i\ge 1$ for some $c>0$, 
we deduce that there exist some constants $C, C'>0$ such that
\begin{align*}
&\int_0^t S(s)^2 ds\le  \gl^{2N} \sum_{i,j=1}^{N-1} \frac{1}{\gamma_i+\gamma_j}|\Phi_{y,i} \Phi_{y,j}|\\
&\le C \lambda^{2N} \Big(\sum_{1\le i \le j \le \frac{N}{\ell}} \frac1{i^2+j^2} + \sum_{1\le i \le \frac{N}{\ell} < j} \frac{N}{j\ell} \frac1{i^2+j^2} + \sum_{\frac{N}{\ell} < i \le j} \frac{N^2}{i j\ell^2} \frac1{i^2+j^2}\Big)\\
&\le C' \lambda^{2N} \log k\;,
\end{align*}
Consequently, setting $\gamma= k^{-\frac{1}{2}-2a}$ and using the fact that from \eqref{smallbias} we have 
$\gamma \,C'\lambda^{2N} \log k < 1$ for $N$ sufficiently large,
we apply \eqref{Eq:ExpoBdMgale} and obtain 
\begin{multline}
\P(N_t^{(t,y)} > \frac 1 2 k^{\frac {1}{2} +a}) 
\le \E[e^{\gamma N_t^{(t,y)}}]e^{-\frac 1 2 \gamma k^{\frac{1}{2}+a}} \\
 \le e^{C' e \gamma^2 k \gl^{2N} \log k -\frac 1 2 \gamma k^{\frac{1}{2}+a}}\le e^{-\frac{1}{4}k^{-a}} \;.
\end{multline}
A similar same computation for $N_t^{(t,y)} <- \frac 1 2 k^{\frac {1}{2} +a}$ and a union bound yield 
\begin{multline}
  \P(\sup_{y \in \lint 1, N/\ell \rint } |\bar H_i(t,y)-\bbE[ H_i(t,y)]| >  \frac 1 2 k^{\frac{1}{2}+a})\\
  =\P(\sup_{y \in \lint 1, N/\ell \rint } |N_t^{(t,y)}| > \frac 1 2  k^{\frac{1}{2}+a}) 
  \le C k^2 e^{-\frac{1}{4}k^{-a}},
  \end{multline}
which combined with \eqref{stimean} allows to conclude.
\end{proof}

\appendix

\section{A monotone grand coupling}\label{Appendix:Coupling}

The construction below is similar to the one detailed in \cite[Section 8.1]{Lac16} in the symmetric case. We consider a collection of independent Poisson clock processes  $\cP^{(i,\ell)}$ and $\cQ^{(i,\ell)}$ with rate $p$ and $q$ respectively where $i\in\lint 1 ,N\rint$ and $\ell \in \lint -N,\ldots,N\rint$: For each $(i,\ell)$, $\cP^{(i,\ell)}$ resp.\ $\cQ^{(i,\ell)}$ is a 
 random increasing sequence of positive real numbers (or equivalently a random locally finite subset of $(0,\infty)$) whose first term and  increments are independent geometric variables of mean $p^{-1}$ resp.\ $q^{-1}$.
 
 For every $k$ and every $\zeta\in \Omega_{N,k}$, we construct the process $(h_t^{\zeta})_{t\ge 0}$ as follows: The process is c\`ad-l\`ag and may  only jump 
  at the times specified by the clock process $\cP$ and $\cQ$.
 We enumerate these Poisson times in increasing order and if
 $t\in \cP^{(i,\ell)}$ and if $h^\zeta_{t-}$ displays a local maximum at $i$ and height $\ell$, that is if
$$h^\zeta_{t-}(i) = \ell = h^\zeta_{t-}(i-1)+1 = h^\zeta_{t-}(i+1)+1\;,$$
then we flip it downwards to a local minimum by setting , $h^\zeta_t(i) := h^\zeta_{t-}(i) - 2$, and $h^{\zeta}_t(j)=h^{\zeta}_{t-}(j)$ for $j\ne i$. 
A similar transition occurs if $Q(i,\ell)$ rings and if $h^\zeta_{t-}$ displays local minimum at $i$ and height $\ell$.\\

It is simple to check that under this construction, $h^\zeta$ indeed evolves according to the right dynamics, and that monotonicity is preserved.

\section{Diffusion bounds for continuous-time supermartingales}\label{sec:diffu}

In this section, we assume that $(M_t)_{t\ge 0}$ is a pure-jump supermartingale with bounded jump rate and jump amplitude. This implies in particular that, $M_t$ is square integrable for all $t>0$. With some abuse of notation, we use the notation $\langle M_\cdot \rangle_t$ for the predictable bracket associated
with the martingale $\tilde M_t=M_t-A_t$ 
where $A$ is the compensator of $M$.

\begin{proposition}\label{prop:solskjaer}
Let $(M_t)_{t\ge 0}$ be as above 
\begin{itemize}
\item[(i)] Set $\tau= \inf\{ t\ge 0 \ : \ M_t=0\}$.
Assume that $M_t$ is non-negative and that, until the absorption time $\tau$, its jump amplitude and jump rate are bounded below by $1$. Then we have for any $a\ge 1$ and all $u>0$
\begin{equation}\label{lopes}
 \bbP[ \tau \ge a^2 u \ | \ M_0\le a ]\le 4  u^{-1/2}.
\end{equation}

\item [(ii)] Given $a \in\bbR$ and $b\le a$, we set  $\tau_{b}:=\inf\{ t\ge 0 \ : \ M_t\le b\}$.
If the amplitude of the jumps of $(M_t)_{t\ge 0}$ is bounded above by $a-b$, 
we have for any $u\ge 0$
\begin{equation}
 \bbP[  \langle M_\cdot \rangle_{\tau_b}\ge (a-b)^2 u  \ | \ M_0\le a ]\le 8  u^{-1/2}.
\end{equation}
\end{itemize}
\end{proposition}
The important building block for the proof of the above proposition is the following result.

\begin{lemma}\label{lem:supersub}
Let $(M_t)_{t\ge 0}$ be as above 
\begin{itemize}
\item[(i)] If the amplitude of the jumps of $(M_t)_{t\ge 0}$ and the jump rate are bounded below by $1$
then for all $\gl \in (0,1)$, 
$$\left(e^{-\gl M_t-\frac{\gl^2 t}{4}}\right)_{t\ge 0}$$
is a submartingale
\item[(ii)] If the amplitude of the jumps of $M_t$ is bounded above by $a$ then for any $\gl\in (0,a^{-1})$ we have
$$\left( \exp\left(-\gl M_t- \frac{\gl^2}{4} \langle M_\cdot \rangle_t\right)\right)_{t\ge 0}$$
is a submartingale.
\end{itemize}
\end{lemma}

\begin{proof}[Proof of Proposition \ref{prop:solskjaer}]

The result only needs to be proved for $u\ge 4$. Without loss of generality one can assume for the proof of both statements  that $P[ M_0\le a ]=1$ and for the second one that $b=0$.
We set $\gl=2 a^{-1} u^{-1/2}$.

\medskip

For $(i)$, a direct application of the Martingale Stopping Theorem to the submartingale of Lemma \ref{lem:supersub} yields: 
\begin{equation}
 \bbE\left[e^{-\frac{\gl^2\tau}{4}}\right]\ge  \bbE[e^{-\gl M_0 }]\ge e^{-\gl a}=e^{-2u^{-1/2}}.
\end{equation}
On the other hand one has 
\begin{equation}
 \bbE\left[e^{-\frac{\gl^2\tau}{4}}\right]\le 1-(1-e^{-\frac{\gl^2}{4} a^2u})\bbP[\tau \ge a^2 u ] \le 1-\frac{1}{2}\bbP[\tau \ge a^2 u].
\end{equation}
The combination of the two yields
 \begin{equation}
  \bbP[\tau \ge a^2 u]\le 2(1-e^{-2u^{-1/2}})\le 4 u^{-1/2}.
 \end{equation}

For $(ii)$, the arguments of the previous case apply almost verbatim if one replaces $\tau$ 
by  $T:=\langle M_\cdot \rangle_{\tau}$. The only thing one has to take into account is that $M_{\tau}$ is not necessarily equal to $0$,
but the assumption on the amplitude of jumps yields $M_{\tau}\ge -a$. The Martingale Stopping Theorem gives us
\begin{equation}
\bbE\left[e^{-\frac{\gl^2 T}{4}}\right]\ge e^{-\gl a } \bbE\left[e^{-\gl M_{\tau}-\frac{\gl^2 T}{4}}\right]\ge e^{-\gl a} \bbE[e^{-\gl M_0}] \ge e^{-4u^{-1/2}}\;.\end{equation}
Repeating the rest of the computation yields
\begin{equation}
\bbP[T \ge a^2 u]\le 2(1-e^{-4u^{-1/2}})\le 8 u^{-1/2}\;.
\end{equation}
\end{proof}
\begin{proof}[Proof of Lemma \ref{lem:supersub}]
Until the end of the proof, we write $\bbE_t$ for the conditional expectation given $(M_s)_{s\le t}$.\\
Case (i). Take $\gl \in (0,1)$. The submartingale identity we need to prove can be written as follows
\begin{equation*}
 \forall s, t\ge 0, \quad  \log \bbE_t\Big[e^{-\gl (M_{t+s}-M_t)}\Big]\ge s\frac{\gl^2}{4},
\end{equation*}
Taking derivative, we deduce that it suffices to prove that for all $t,s\ge 0$ we have
\begin{equation}\label{toprov}
 \lim_{h\downarrow 0} \frac1{h}\bbE_t\Big[e^{-\gl( M_{t+s+h}-M_t)}-e^{-\gl (M_{t+s}-M_t)}\Big] \ge \frac{\gl^2}{4}  \bbE_t\Big[e^{-\gl (M_{t+s}-M_t)}\Big]\;.
 \end{equation}
Notice that for all $x\in \bbR$ 
\begin{equation}\label{Eq:ExpoTaylor}
e^{-x}+x-1\ge  \frac{\min(1,x^{2})}{4}\;.
\end{equation}
Thus, using the supermartingale property of $M$ we have for all $\gl\in(0,1)$
\begin{align*}
&\bbE_t\Big[e^{-\gl (M_{t+s+h}-M_t)}-e^{-\gl (M_{t+s}-M_t)}\Big]\\
&= \bbE_t\Big[e^{-\gl (M_{t+s}-M_t)}\bbE_{t+s}\big[e^{-\gl (M_{t+s+h}-M_{t+s})}-1\big]\Big]\\
&\ge \bbE_t\Big[e^{-\gl (M_{t+s}-M_t)}\bbE_{t+s}\big[e^{-\gl (M_{t+s+h}-M_{t+s})}+\gl (M_{t+s+h}-M_{t+s})-1\big]\Big]\\
&\ge \frac{\gl^2}{4}\bbE_t\Big[e^{-\gl (M_{t+s}-M_t)}\bbE_{t+s}\big[\min(1,(M_{t+s+h}-M_{t+s})^2)\big]\Big]\;.
\end{align*}
The assumption on the jump rates and the jump amplitudes yield
\begin{equation}
\liminf_{h\to 0} \frac1{h}\bbE_{t+s}\big[\min(1,(M_{t+s+h}-M_{t+s})^2)\big]\ge 1,
\end{equation}
so that Fatou's Lemma concludes the proof.

\medskip

Case (ii). We can assume without loss of generality that $a=1$. 
Here again, taking the derivative of the submartingale identity that we want to establish, it suffices to prove that for all $t,s \ge 0$ we have
\begin{align*}
\liminf_{h\downarrow 0} &\frac1{h}\bbE_t\Big[e^{-\gl M_{t+s+h} - \frac{\gl^2}{4}\langle M_\cdot\rangle_{t+s+h}} -e^{-\gl M_{t+s}- \frac{\gl^2}{4}\langle M_\cdot\rangle_{t+s}} \Big]
\ge  0\;.\end{align*}
Taking the conditional expectation w.r.t. $M_{t+s}$, we see that it suffices to prove the existence of some deterministic constant $C>0$ such that
\begin{equation}\label{Eq:subderiv} \bbE_{t+s}\Big[e^{-\gl (M_{t+s+h}-M_{t+s}) - \frac{\gl^2}{4}(\langle M_\cdot\rangle_{t+s+h}-\langle M_\cdot\rangle_{t+s})} -1 \Big] \ge -C h^2\;,
\end{equation}
for all $h$ small enough.\\
Without loss of generality, we can assume that $t+s = 0$ and $M_0=0$. Recall that $\tilde M\ge M$ is the martingale which is obtained by subtracting the (negative) compensator. Thus
\begin{align*}
e^{-\gl M_{h} - \frac{\gl^2}{4} \langle M_\cdot\rangle_{h}} -1 &\ge e^{-\gl \tilde{M}_{h} - \frac{\gl^2}{4} \langle M_\cdot\rangle_{h}} -1\\
&\ge \big(1-\gl \tilde M_{h}+ \frac{1}{4}\min(1, \gl^{2} \tilde M_{h}^2)\big) \big(1- \frac{\gl^2}{4} \langle M_\cdot\rangle_{h}\big)-1 \\
&\ge -\gl \tilde M_{h}+  \frac{\gl^2}{4}\big( \tilde M_{h}^2- \langle M_\cdot\rangle_{h}\big)
- \frac{\gl^2}{4}\big(\tilde M_{h}^2-\gl^{-2}\big)_+ \\
&\quad+ \frac{\gl^2}{4}\langle M_\cdot\rangle_{h} \big(\gl \tilde M_{h} - \frac{1}{4}\min(1, \gl^2 \tilde M_{h}^2) \big)\;,
\end{align*}
so that
\begin{align*}
\bbE\Big[e^{-\gl M_{h} - \frac{\gl^2}{4} \langle M_\cdot\rangle_{h}} -1\Big] &\ge  \bbE\Big[- \frac{\gl^2}{4}\big(\tilde M_{h}^2-\gl^{-2}\big)_+\\
&\quad+ \frac{\gl^2}{4}\langle M_\cdot\rangle_{h} \big(\gl \tilde M_{h} - \frac{1}{4}\min(1, \gl^2 \tilde M_{h}^2) \big)\Big]\;.
\end{align*}
Take $\gl \in (0,1)$. Our assumptions on the increments and jump rates imply that for some constant $C>0$ we have
\begin{equation}\begin{split}
\bbE \left [ \big(\tilde M_{h}^2-\gl^{-2}\big)_+ \right]&\le C h^2,\\ \langle M_\cdot\rangle_{h} &\le C h,\\
 \max\left( \bbE\left[ |\tilde M_{h}| \right], \bbE \left[  \min(1, \gl^2 \tilde M_{h}^2)  \right]\right)&\le Ch,
 \end{split}
 \end{equation} 
 (the compensator being of order $h$ the estimates for $\tilde M$ can be deduced from that for $M$), which allows to conclude that \eqref{Eq:subderiv} holds.

\end{proof}

\section{Exponential moments of continuous-time martingales}\label{lapC}

Let $(M_t)_{t\ge 0}$ be a martingale defined as a function of a continuous time Markov chain on a finite state space
$$M_t=f(t,X_t)\;,$$
where $f$ is differentiable in time.
We let $B$ denote the maximal jump rate for $X$ and let $S(t)$ denote the maximal amplitude for a jump of $M$ at time $t$:
$$S(t):=\max_{\xi\sim \xi'} |f(t,\xi)-f(t,\xi')|\;.$$

\begin{lemma}\label{lem:expobd}
For any $\gl > 0$ we have 
\begin{equation}
\bbE \left[ e^{\gl M_t} \right] \le \exp \left( B\int^t_0 \left[e^{\gl S(s)}-\gl S(s)-1\right]\dd s \right).
\end{equation}
In particular if $\gl S(t)\le 1$ for all $t\ge 0$ then we have 
\begin{equation}\label{Eq:ExpoBdMgale}
\bbE \left[ e^{\gl M_t} \right] \le \exp \left( B e \gl^2 \int^t_0 S^2(s) \dd s \right).
\end{equation}
\end{lemma}

\begin{proof}

We are going to show that for all $t\ge 0$
\begin{equation}
\partial_t \log \bbE \left[ e^{\gl M_t} \right]= \frac{\partial_t \bbE \left[ e^{\gl M_t} \right]}{\bbE \left[ e^{\gl M_t} \right]}\le B [e^{\gl S(t)}-\gl S(t)-1].
\end{equation}
To that end, it is sufficient to show that almost surely
\begin{equation}
\partial_s \bbE \left[ e^{\gl (M_{t+s}-M_t)} - 1  \ | \ \cF_t \right] |_{s=0} \le B [e^{\gl S(t)}-\gl S(t)-1].
\end{equation}
We let $\Delta_s M=M_{t+s}-M_t$ denote the martingale increment and as in the previous section write $\bbE_t$ for the conditional expectation w.r.t.\ $M_t$ . By the martingale property, we have
\begin{align*}
\bbE_t \left[ e^{\gl \Delta_s M}-1 \right]&=\bbE_t \left[ e^{\gl \Delta_s M}-\gl \Delta_s M-1 \right] \le \bbE_t \left[ e^{\gl |\Delta_s M|}-\gl |\Delta_s M|-1 \right].
\end{align*}
Note that $|\Delta_s M|$ is stochastically dominated by 
$$\left[\max_{u\in[t,t+s]} S(u)\right] \cW+ s\times \max_{u\in[t,t+s]}  \| \partial_u f(u, \cdot)\|_{\infty} $$
 where $\cW$ is a Poisson variable of parameter $B s$. As $S$ is Lipshitz we conclude that 
\begin{equation}
\bbE_t \left[ \frac{e^{\gl |\Delta_s M|}-\gl |\Delta_s M|-1}{s} \right]\le  B[e^{\gl S(t)}-\gl S(t)-1]+ c s.
\end{equation}
\end{proof}

\bibliographystyle{Martin}
\bibliography{library}

\begin{thebibliography}{BBHM05}
\expandafter\ifx\csname url\endcsname\relax
  \def\url#1{\texttt{#1}}\fi
\expandafter\ifx\csname urlprefix\endcsname\relax\def\urlprefix{URL }\fi
\expandafter\ifx\csname href\endcsname\relax
  \def\href#1#2{#2}\fi
\expandafter\ifx\csname burlalt\endcsname\relax
  \def\burlalt#1#2{\href{#2}{\texttt{#1}}}\fi

\bibitem[AD86]{AldDia}
\textsc{D.~Aldous} and \textsc{P.~Diaconis}.
\newblock Shuffling cards and stopping times.
\newblock \emph{Amer. Math. Monthly} \textbf{93}, no.~5, (1986), 333--348.

\bibitem[BBHM05]{Benjamini}
\textsc{I.~Benjamini}, \textsc{N.~Berger}, \textsc{C.~Hoffman}, and
  \textsc{E.~Mossel}.
\newblock Mixing times of the biased card shuffling and the asymmetric
  exclusion process.
\newblock \emph{Trans. Amer. Math. Soc.} \textbf{357}, no.~8, (2005),
  3013--3029 (electronic).
\newblock
  \burlalt{doi:10.1090/S0002-9947-05-03610-X}{http://dx.doi.org/10.1090/S0002-9947-05-03610-X}.

\bibitem[BG97]{BG97}
\textsc{L.~Bertini} and \textsc{G.~Giacomin}.
\newblock Stochastic {B}urgers and {KPZ} equations from particle systems.
\newblock \emph{Comm. Math. Phys.} \textbf{183}, no.~3, (1997), 571--607.
\newblock
  \burlalt{doi:10.1007/s002200050044}{http://dx.doi.org/10.1007/s002200050044}.

\bibitem[DG91]{DG91}
\textsc{P.~Dittrich} and \textsc{J.~G\"artner}.
\newblock A central limit theorem for the weakly asymmetric simple exclusion
  process.
\newblock \emph{Math. Nachr.} \textbf{151}, (1991), 75--93.
\newblock
  \burlalt{doi:10.1002/mana.19911510107}{http://dx.doi.org/10.1002/mana.19911510107}.

\bibitem[DMPS89]{Demasi89}
\textsc{A.~De~Masi}, \textsc{E.~Presutti}, and \textsc{E.~Scacciatelli}.
\newblock The weakly asymmetric simple exclusion process.
\newblock \emph{Ann. Inst. H. Poincar\'e Probab. Statist.} \textbf{25}, no.~1,
  (1989), 1--38.

\bibitem[DS81]{DiaSha}
\textsc{P.~Diaconis} and \textsc{M.~Shahshahani}.
\newblock Generating a random permutation with random transpositions.
\newblock \emph{Z. Wahrsch. Verw. Gebiete} \textbf{57}, no.~2, (1981),
  159--179.

\bibitem[G{\"a}r88]{Gartner88}
\textsc{J.~G{\"a}rtner}.
\newblock Convergence towards {B}urgers' equation and propagation of chaos for
  weakly asymmetric exclusion processes.
\newblock \emph{Stochastic Process. Appl.} \textbf{27}, no.~2, (1988),
  233--260.
\newblock
  \burlalt{doi:10.1016/0304-4149(87)90040-8}{http://dx.doi.org/10.1016/0304-4149(87)90040-8}.

\bibitem[KL99]{KipLan}
\textsc{C.~Kipnis} and \textsc{C.~Landim}.
\newblock \emph{Scaling limits of interacting particle systems}, vol. 320 of
  \emph{Grundlehren der Mathematischen Wissenschaften [Fundamental Principles
  of Mathematical Sciences]}.
\newblock Springer-Verlag, Berlin, 1999.

\bibitem[KOV89]{KOV}
\textsc{C.~Kipnis}, \textsc{S.~Olla}, and \textsc{S.~R.~S. Varadhan}.
\newblock Hydrodynamics and large deviation for simple exclusion processes.
\newblock \emph{Comm. Pure Appl. Math.} \textbf{42}, no.~2, (1989), 115--137.
\newblock
  \burlalt{doi:10.1002/cpa.3160420202}{http://dx.doi.org/10.1002/cpa.3160420202}.

\bibitem[Lab17]{LabbeKPZ}
\textsc{C.~Labb\'e}.
\newblock Weakly asymmetric bridges and the {KPZ} equation.
\newblock \emph{Comm. Math. Phys.} \textbf{353}, no.~3, (2017), 1261--1298.
\newblock
  \burlalt{doi:10.1007/s00220-017-2875-0}{http://dx.doi.org/10.1007/s00220-017-2875-0}.

\bibitem[Lac16a]{Lac162}
\textsc{H.~Lacoin}.
\newblock The cutoff profile for the simple exclusion process on the circle.
\newblock \emph{Ann. Probab.} \textbf{44}, no.~5, (2016), 3399--3430.
\newblock
  \burlalt{doi:10.1214/15-AOP1053}{http://dx.doi.org/10.1214/15-AOP1053}.

\bibitem[Lac16b]{Lac16}
\textsc{H.~Lacoin}.
\newblock Mixing time and cutoff for the adjacent transposition shuffle and the
  simple exclusion.
\newblock \emph{Ann. Probab.} \textbf{44}, no.~2, (2016), 1426--1487.
\newblock
  \burlalt{doi:10.1214/15-AOP1004}{http://dx.doi.org/10.1214/15-AOP1004}.

\bibitem[Lig05]{Liggett}
\textsc{T.~M. Liggett}.
\newblock \emph{Interacting particle systems}.
\newblock Classics in Mathematics. Springer-Verlag, Berlin, 2005.
\newblock Reprint of the 1985 original.

\bibitem[LL18]{LabLac16}
\textsc{C.~Labb\'e} and \textsc{H.~Lacoin}.
\newblock {Cutoff phenomenon for the asymmetric simple exclusion process and
  the biased card shuffling}.
\newblock \emph{Ann. Probab.} (2018+).

\bibitem[LP16]{LevPer16}
\textsc{D.~A. Levin} and \textsc{Y.~Peres}.
\newblock Mixing of the exclusion process with small bias.
\newblock \emph{J. Stat. Phys.} \textbf{165}, no.~6, (2016), 1036--1050.
\newblock
  \burlalt{doi:10.1007/s10955-016-1664-z}{http://dx.doi.org/10.1007/s10955-016-1664-z}.

\bibitem[LPW17]{LevPerWil}
\textsc{D.~A. Levin}, \textsc{Y.~Peres}, and \textsc{E.~L. Wilmer}.
\newblock \emph{Markov chains and mixing times}.
\newblock American Mathematical Society, Providence, RI, 2017.
\newblock Second edition of [ MR2466937], With a chapter on ``Coupling from the
  past'' by James G. Propp and David B. Wilson.

\bibitem[Mor06]{Morris06}
\textsc{B.~Morris}.
\newblock The mixing time for simple exclusion.
\newblock \emph{Ann. Appl. Probab.} \textbf{16}, no.~2, (2006), 615--635.
\newblock
  \burlalt{doi:10.1214/105051605000000728}{http://dx.doi.org/10.1214/105051605000000728}.

\bibitem[Rez91]{Reza}
\textsc{F.~Rezakhanlou}.
\newblock Hydrodynamic limit for attractive particle systems on {${\bf Z}^d$}.
\newblock \emph{Comm. Math. Phys.} \textbf{140}, no.~3, (1991), 417--448.

\bibitem[Wil04]{Wil04}
\textsc{D.~B. Wilson}.
\newblock Mixing times of {L}ozenge tiling and card shuffling {M}arkov chains.
\newblock \emph{Ann. Appl. Probab.} \textbf{14}, no.~1, (2004), 274--325.
\newblock
  \burlalt{doi:10.1214/aoap/1075828054}{http://dx.doi.org/10.1214/aoap/1075828054}.

\end{thebibliography}

\end{document}